\documentclass[11pt]{article}
\usepackage{txmac,times,case,epsf}
\usepackage{tocloft}
\usepackage{authblk}

\usepackage{graphicx,amssymb}
\usepackage{epsfig}
\usepackage[margin=0.8in]{geometry}
\usepackage[...]{youngtab}
\usepackage{ytableau}

\usepackage{tikz}
\usetikzlibrary{decorations.pathreplacing,decorations.markings}
\usetikzlibrary{braids,calc}
\usepackage{graphicx}

\DeclareGraphicsExtensions{.pdf,.png,.jpg}
\usepackage{amsmath,amssymb,amsthm,mathtools}
\usepackage{stmaryrd}    
\usepackage{xfrac}
\DeclareMathAlphabet{\mathpzc}{OT1}{pzc}{m}{it} 
\newcommand\mathscr[1]{\scalebox{1.1}{$\mathpzc{#1}$}}
\usepackage{booktabs} 
\usepackage[all]{xy}
\usepackage{tikz}
\usetikzlibrary{cd}

\usepackage{float}
\usepackage{xcolor}
\usepackage{rotating}
\usepackage{multirow}
\usepackage{colortbl}
\usepackage{amsthm}
\usepackage{amssymb}
\pgfdeclarelayer{background}
\pgfdeclarelayer{foreground}
\pgfsetlayers{background,main,foreground}
\usepackage{hyperref}
\hypersetup{
  colorlinks=true,
  allcolors=.
}
\usepackage{cleveref}
\crefformat{section}{\S#2#1#3}
\crefformat{subsection}{\S#2#1#3}
\crefformat{subsubsection}{\S#2#1#3}

\usepackage{mathptmx}

\def\be{\begin{eqnarray}}
\def\ee{\end{eqnarray}}
\def\nn{\nonumber}

\makeatletter

\let\sg\sigma
\def\bt{\beta}
\let\q\quad
\let\ni\noindent
\let\sS\subset
\let\ti\times
\let\pa\partial
\def\Qp{Quasipositivity}
\def\qp{quasipositive}
\def\QP{quasipositivity}
\def\SP{strong quasipositivity}
\def\sp{strongly quasipositive}
\let\es\enspace
\let\th\theta
\let\kp\kappa
\let\nb\nabla
\let\dl\delta
\let\eps\varepsilon
\let\ul\underline
\let\ol\overline
\def\ob{\overbrace}
\def\ub{\underbrace}
\def\md{\min\deg}
\def\Md{\max\deg}
\def\mcf{\min{\operator@font cf}\,}
\def\Mc{\max{\operator@font cf}\,}
\let\mc\mcf
\def\sgn{{\operator@font sgn}}
\def\MFW{\mathop {\operator@font MFW}\mathord{\!}}
\def\spn{\mathop {\operator@font span}\mathord{}}
\def\br#1{\left\lfloor#1\right\rfloor}
\def\BR#1{\left\lceil#1\right\rceil}
\let\wt\widetilde
\def\tK{\wt K}
\def\tS{\wt S}
\def\eqref#1{\mbox{(\protect\ref{#1}})}
\def\bZ{{\Bbb Z}}
\let\Dl\Delta
\let\gm\gamma
\let\lm\lambda
\let\ap\alpha
\let\io\iota
\def\uM{\underline{M}}

\newenvironment{eqn}{\begin{equation}}{\end{equation}\@ignoretrue}

\newif\if@sup
\newtoks\@sups
\def\append@sup#1{\edef\act{\noexpand\@sups={\the\@sups #1}}\act}%
\def\reset@sup{\@supfalse\@sups={}}%
\def\mk@scripts#1#2{\if #2/ \if@sup ^{\the\@sups}\fi \else%
  \ifx #1_ \if@sup ^{\the\@sups}\reset@sup \fi {}_{#2}%
  \else \append@sup#2 \@suptrue \fi%
  \expandafter\mk@scripts\fi}
\def\tensor#1#2{\reset@sup#1\mk@scripts#2_/}
\def\multiscripts#1#2#3{\reset@sup{}\mk@scripts#1_/#2%
  \reset@sup\mk@scripts#3_/}
\makeatother

\makeatletter
\newbox\slashbox \setbox\slashbox=\hbox{$/$}
\def\itex@pslash#1{\setbox\@tempboxa=\hbox{$#1$}
  \@tempdima=0.5\wd\slashbox \advance\@tempdima 0.5\wd\@tempboxa
  \copy\slashbox \kern-\@tempdima \box\@tempboxa}
\def\slash{\protect\itex@pslash}
\makeatother

\let\PLAINthebibliography\thebibliography
\renewcommand\thebibliography[1]{
  \PLAINthebibliography{#1}
  \setlength{\parskip}{0.7pt}
  \setlength{\itemsep}{0.7pt plus .3ex}
}

\usepackage[titletoc]{appendix}
\newtheorem{theorem}{Theorem}[section]
\newtheorem{lemma}[theorem]{Lemma}
\newtheorem{prop}[theorem]{Proposition}
\newtheorem{corr}[theorem]{Corollary}

\theoremstyle{definition}
\newtheorem{defn}[theorem]{Definition}
\newtheorem{notation}[theorem]{Notation}
\newtheorem{example}[theorem]{Example}
\newtheorem{remark}[theorem]{Remark}

\newtheorem{ques}[theorem]{Question}
\newtheorem{prob}[theorem]{Problem}

\title{Panhandle polynomials of torus links and geometric applications}
\author[1,2,3]{Andrei Mironov\thanks{mironov@itep.ru}}
\author[4,5]{Hisham Sati\thanks{hsati@nyu.edu}}
\author[4]{Vivek Kumar Singh\thanks{vks2024@nyu.edu}}
\author[6]{Alexander Stoimenov\thanks{stoimeno@stoimenov.net}}
\affil[1]{Lebedev Physics Institute, Moscow 119991, Russia.}
\affil[2]{NRC Kurchatov Institute 123182, Moscow, Russia.}
\affil[3]{Institute for Information Transmission Problems, Moscow 127994, Russia.}
\affil[4]{Center for Quantum and Topological Systems (CQTS), NYUAD Research Institute,  
New York University Abu Dhabi, PO Box 129188, Abu Dhabi, UAE }
\affil[5]{The Courant Institute for Mathematical Sciences, NYU, NY.}
\affil[6]{ Dongguk University, WISE campus, Department of Mathematics Education,
123, Dongdae-ro,
38066 Gyeongju-si, Republic
of Korea}

\begin{document}

\maketitle

\begin{abstract}
We use a decomposition of the tensor of the fundamental 
representation of the quantum group $U_q(\mathfrak{sl}_N)$ and the 
Rosso-Jones formula to establish a peculiar ``panhandle''
shape of the HOMFLY-PT polynomial of the reverse parallel of
torus knots and links. Due to their panhandle-like intrinsic
properties, the HOMFLY-PT polynomial is referred to as a
``panhandle polynomial''. With the help of the $\ell$-invariant,
this extends to links the Etnyre-Honda result about
the arc index and maximal Thurston-Bennequin invariant of
torus knots. It has further geometric consequences, related
to the braid index, the existence of minimal string Bennequin 
surfaces for banded and Whitehead doubled links, the Bennequin
sharpness problem, and the equivalence of their quasipositivity and strong quasipositivity.
We extend these properties to torus links,
which relate to the classification of their component-wise
Thurston-Bennequin invariants.
Finally, we discuss the definition of the $\ell$-invariant
for general links.
\end{abstract}
    
\newpage 

\mbox{}\vspace{3cm}\mbox{}
{
\tableofcontents 
}\mbox{}

\newpage 

\section{Introduction}

Polynomial invariants have provided deep insights into knot theory and topology.
Some of the most widely studied knot polynomial invariants include the Alexander polynomial~\cite{Alexander1928}, the Jones polynomial~\cite{Jones1985}, and the HOMFLY-PT polynomial~\cite{HOMFLY1985}. 
In this paper, we exhibit a peculiar shape
of the latter polynomial 
of the reverse parallel of a torus knot and a torus link.
We fix for $K=T_{m,n}=T(m,n)$, with
\[
(m,n)=1\,,\q m<n\,,
\]
the vertical framing $t=t_\nu$, given in \eqref{tnu}, 
and the convention \eqref{skrel} below. 
    Our main result is the following:

\begin{theorem}[Panhandle Theorem]\label{MYTH}
 Let $X = [\uM; M]_v$ denote a polynomial such that 
    \vspace{-2mm} 
\begin{eqn}\label{Mm}
\min\deg_v X = \uM \quad\text{and}\quad \max\deg_v X = M.
\end{eqn}

\vspace{-2mm} 
\noindent Then the HOMFLY-PT polynomial for the reverse $2$- cable torus knot $C_2(T_{m,n},t)$ has the
form
\vspace{-2mm} 
\[P(C_2(T_{m,n},t))=[1-2m; 2m-1]_v
\; +\; 
\underbrace{(m-1) z 
v \frac{
v^{2n}-
v^{2m}}{
v^{2}-1}}_{\rm{panhandle}}\,.\]
\end{theorem}

The condition $m<n $ is essential: restoring the symmetry of the torus knot $T(m,n)$ in $m$ and $n$ (topological framing) requires the vanishing of certain leading coefficients in the Laurent polynomial $[1-2m; 2m-1]_v$ when $m>n$.

\smallskip 
Establishing Theorem \ref{MYTH} directly from the skein relation \eqref{skrel} is cumbersome. Hence, we adopt a representation theory approach.
The construction of HOMFLY-PT invariants, denoted by $\mathcal{H}(C_2(T_{m,n},t))$,  naturally relies on the 
decomposition of the tensor product of representations (see \cite{RT}\cite{MMM12})
into irreducible components
\[
R \otimes \bar{R} \;=\; \bigoplus_i Q_i, \qquad Q_i \in \mathrm{Rep}(U_q(\mathfrak{sl}_N))\,.
\]

\vspace{-2mm} 
\noindent Here $R$ denotes the fundamental representation associated with a Young tableau of the quantum group $U_q(\mathfrak{sl}_N)$ and $\bar R$ is its conjugate.
In this context, the HOMFLY-PT polynomial of a knot $C_{2}(T_{m,n},t)$ can be written as the sum of contributions from the relevant representations, namely the scalar representation ($Q_{1}=\Phi$) and the adjoint representation ($Q_{2}=\mathrm{Adj}$):
\vspace{-3mm}
\begin{equation}{\label{RepHOM}}
P\big(C_{2}(T_{m,n},t_\nu)\big)
: =\mathcal{H}\big(C_2(T_{m,n},t_\nu)\big)= 
\frac{1}{\mathcal{H}_{R}({U})}\Big(\mathcal{H}_{\Phi}\big({T_{m,n}}\big) + \mathcal{H}_{\mathrm{Adj}}\big({T_{m,n}}\big)\!\Big).
\end{equation}
Here 
\begin{eqn}\label{tnu}
t_\nu=(1-m)n
\end{eqn}
is the vertical framing (see Defs.~\ref{r2cable} and ~\ref{2cabledhom}) and $U$ is the
unknot. 
Whereas the HOMFLY-PT polynomial of the scalar representation $\mathcal{H}_{\Phi}(K)=1$, 
that of the adjoint representation $\mathcal{H}_{\mathrm{Adj}}(K)$ is generally nontrivial. 
This formulation provides a natural framework for analyzing the structure of $P\big(C_{2}(T_{m,n},t)\big)$.
Theorem \ref{MYTH} emerges from expression \eqref{RepHOM} using the Rosso-Jones formula 
(\cite[Theorem 5.1]{LinZheng2006}\cite[Theorem 8]{RJ}), which provides a systematic approach to understanding the adjoint polynomials of torus knots $\mathcal{H}_{\mathrm{Adj}}({T_{m,n}})$.
We describe this in \cref{Sec-Panhandle}, with a general proof of Theorem \ref{MYTH} established in \cref{Sec-proof}. 
The particular case of $m=2$ of Theorem \ref{MYTH} was proved in 
\cite{part2} by the fourth author using skein algebra programming.
The leading ($v$-degree) term of the panhandle for $m=2$ was also
identified by Diao and Morton \cite[Theorem 2.7]{DM}.
 
\smallskip
The main motivation behind Theorem \ref{MYTH} is not the peculiar shape of the polynomial.
It emerged through explicit computations related to very different subjects.
Specifically, it is related to a new method of estimating the arc index and 
maximal Thurston-Bennequin invariant of knots \cite{part2}.
This method, which relies on a quantity named $\ell$-invariant,
allows for an algebraic proof
of some of Etnyre-Honda's work \cite{EH} on torus knots.
The work in \cite{part2} allows us to further relate to
the braid index of the reverse parallel (see \cite{DM}),
the existence of 
its minimal string Bennequin surfaces (see \cite{BMe}\cite{HS}),
the Bennequin sharpness problem (Problem \ref{BSP}),
and the equivalence of \QP\ and \SP\ for certain links.
We outline these applications in \S\ref{S3}.

\smallskip 
Theorem \ref{MYTH} also admits an extension to the setting of torus links, which we discuss in \cref{Sec-ExtensionToLinks}. This extension follows directly from the underlying Def.~\ref{2cabledhomL} and the arguments developed in \S\ref{Sec-ExtensionToLinks}. We state the result in the following theorem (see Theorem \ref{PHT}).

\begin{theorem}[Panhandle for links]
\label{PHT1}
The HOMFLY-PT polynomial for the reverse $2$-cable $l$-component torus link $C_2(T_{m,n},t_\nu)$ has the form 
\vspace{-2mm} 
$$
P\big(C_2(T_{m,n},t_{\nu})\big)
=\bigg[1-2m; \, 2n\ {l-1\over l}\ +\frac{2m}{l}-1\!\bigg]\bigg]_{_v} 
\; +\; 
\underbrace{\zeta_{m,l}\ z
v\ \frac{
v^{2n}-
v^{2m}}{
v^{2}-1}}_{\rm{panhandle}}\,,
$$
\vspace{-3mm} 
where
$$
\zeta_{m,l}:=\sum_{k=0}^l(-1)^{l+k}{l!\over (l-k)!}\left({m\over l}\right)^k.
$$
Hence, the length of the link panhandle is equal to $2(n-m)/l$.
\end{theorem}

\smallskip 
In \S\ref{S6} we extend the above applications to torus links.
Apart from determining their arc index, we give a
full description of their component-wise Thurston-Bennequin invariants
in \S\ref{STB}. Then, the further geometric applications
outlined for torus knots are largely generalized to torus links
in \S\ref{SGT}, with the addition of the Baker-Motegi problem
(\S\ref{SBMP}). The conclusion of \S\ref{S6} (and the paper)
discusses the extension of the $\ell$-invariant for general
links, based on the properties studied for the torus links.

\section{ Panhandle polynomials for torus knots and links} 
\label{Sec-Panhandle}

\subsection{Background and Definitions}
\begin{defn}[Torus Knots]
\label{def11}
    Let \( m, n \in \mathbb{Z}_{>0} \) such that \( \gcd(m,n) = 1 \). The \emph{torus knot} \( T_{m,n}=T(m,n) \) is defined as the closure of the braid word
\[
T_{m,n} := \widehat{(\sigma_1 \sigma_2 \cdots \sigma_{m-1})^n} \in B_m,
\]
where \( B_m \) is the braid group on \( m \) strands with standard Artin generators \( \sigma_1, \dots, \sigma_{m-1} \). The closure operation, $\widehat{\cdot}$\,, connects the corresponding top and bottom endpoints of the braid to form a knot.

\begin{example}
    
The torus knot \( T_{4,3} \) is the closure of the braid \( (\sigma_1 \sigma_2 \sigma_3)^3 \in B_4 \). See Fig.~\ref{fig:T43} for an illustration.
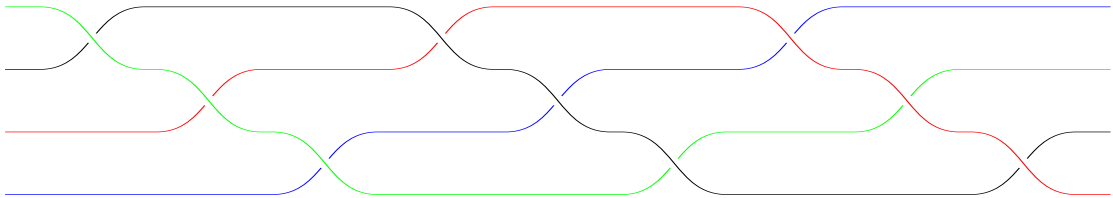
\begin{figure}[ht]
\centering
\resizebox{15cm}{2.5cm}{
\begin{tikzpicture}

\pic[rotate=90,
braid/.cd,
every strand/.style={ultra thin},
strand 1/.style={blue},
strand 2/.style={red},
strand 3/.style={black},
strand 4/.style={green},
xshift=-1.2pt
] {braid={ s_3^{-1} s_2^{-1} s_1^{-1}s_3^{-1} s_2^{-1}s_1^{-1} s_3^{-1} s_2^{-1}s_1^{-1}}};

\end{tikzpicture}
}
\caption{  Braid representation of torus knots $T(4,3)$.}
\label{fig:T43}
\end{figure}
\end{example}
\end{defn}

\begin{defn}[{The reverse $2$-cable knots}]
{\label{r2cable}}
Let \({K} = \widehat{\beta}\) be the closure of a braid \(\beta\in B_m\).  
Define the \emph{reverse $2$-cable} of \(\beta\) to be the braid \( \phi(\beta)\in B_{2m} \) obtained by ``doubling and reversing orientation'' of each strand of \(\beta\) with  framing $t$. The framing $t\in\mathbb{Z}$ is defined by the factor of $q^{t\kappa_\lambda}$ (see \eqref{Cas}) that differs the framed HOMFLY-PT polynomial from the invariant polynomial (topological framing). The sign of framing changes if the orientation is reversed. We will use the vertical framing \cite{AM}, which for the torus knot \( T(m, n) \) is given in \eqref{tnu}. Concretely, on the Artin generators
\begin{equation}
\label{Artin-gen} 
\phi: B_m \;\longrightarrow\; B_{2m}, 
\qquad
\phi\bigl(\sigma_i\bigr) \;=\; \sigma_{2i} \sigma_{2i+1}\sigma_{2i-1}\,\sigma_{2i}
\quad (1\le i\le m-1).
\end{equation}
Then the \emph{2-cable knot} \(C_2({K},t)\) is
\[
C_2({K},t)\;:=\;\widehat{\phi(\beta)} \;\in\; S^3.
\]
See Fig.~\ref{4-torus} for a schematic of the \emph{reverse $2$-cable torus knot $T_{4,5}$} with framing $t=-15$.

\smallskip 
\begin{figure}[ht]
\centering
\resizebox{15cm}{2.5cm}{
\begin{tikzpicture}

\pic[rotate=90,
braid/.cd,
every strand/.style={ultra thin},
strand 1/.style={blue},
strand 2/.style={red},
strand 3/.style={black},
strand 4/.style={green},
xshift=-1.4pt
] {braid={ s_3^{-1} s_2^{-1} s_1^{-1}s_3^{-1} s_2^{-1} s_1^{-1}s_3^{-1} s_2^{-1} s_1^{-1}s_3^{-1} s_2^{-1} s_1^{-1}s_3^{-1} s_2^{-1} s_1^{-1} }};

\pic[rotate=90,
braid/.cd,
every strand/.style={ultra thin},
strand 1/.style={blue},
strand 2/.style={red},
strand 3/.style={black},
strand 4/.style={green},
strand 4/.style={green},
xshift=1.4pt
] {braid={ s_3^{-1} s_2^{-1} s_1^{-1}s_3^{-1} s_2^{-1} s_1^{-1}s_3^{-1} s_2^{-1} s_1^{-1}s_3^{-1} s_2^{-1} s_1^{-1}s_3^{-1} s_2^{-1} s_1^{-1} }};

\draw[->, blue, thin] (0.0,-0.05) -- ++(0.2,0); 
\draw[<-, blue, thin] (0.0,0.05) -- ++(0.2,0);
\draw[<-, red, thin] (0.0,1.05) -- ++(0.2,0); 
\draw[->, red, thin] (0.0,.95) -- ++(0.2,0);
\draw[<-, black, thin] (0.0,2.05) -- ++(0.2,0); 
\draw[->, black, thin] (0.0,1.95) -- ++(0.2,0);
\draw[<-, green, thin] (0.0,3.05) -- ++(0.2,0); 
\draw[->, green, thin] (0.0,2.95) -- ++(0.2,0);

\end{tikzpicture}
}
\caption{Braid representation of reverse $2$ -cable torus knots $T_{4,5}$.}
\label{4-torus}
\end{figure}
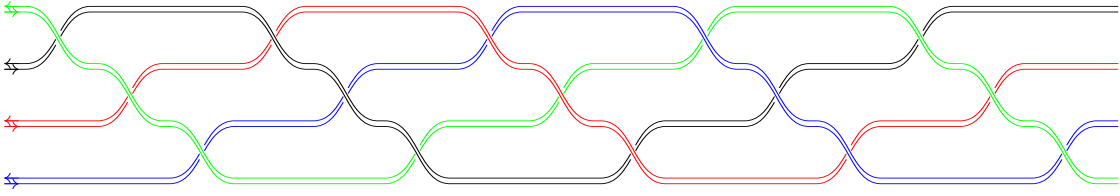
\end{defn}

\begin{defn}[HOMFLY-PT Polynomial\cite{HOMFLY1985}]{\label{HOM}}
For a knot $K$, the (uncolored) \textit{HOMFLY--PT polynomial} $P(K)$ is defined via the skein relation (Morton convention):
\vspace{-2mm} 
\begin{equation}\label{skrel}
v^{-1}\,P(L_+) - 
v P(L_-) = z\,P(L_0), ~~~~ P({U})=1.
\end{equation}

\vspace{-2mm} 
\noindent Here, ${U}$ denotes the unknot (trivial knot), and $z = q - q^{-1}$.
\end{defn}

As highlighted in the Introduction, we focus on studying the HOMFLY-PT polynomial of \emph{reverse $2$-cable torus knots} and its geometric applications. Directly applying the skein relations (see Def.~\ref{HOM}) can be quite complex in general. Therefore, we address the problem through representation theory\cite{RT}, focusing on the 
$R$-colored HOMFLY-PT invariants $H_R(K)$, where the Young tableau $R$ 
corresponds to a representation of $U_q(\mathfrak{sl}_N)$(see Def.~\ref{YTDR}). In particular, 
when $R=[1]$ (the fundamental representation), the invariant reduces to the 
usual HOMFLY-PT polynomial \eqref{skrel}, i.e.,
\vspace{-2mm} 
\[
H_{[1]}(K) \;=\; H(K) \;=\; P(K).
\]
\begin{remark}[Normalizations] There are two normalizations of the HOMFLY-PT polynomial: the normalized (reduced) polynomial such that $H_R({U})=1$ where ${U}$ denotes the unknot (trivial knot), and the un-normalized (unreduced) polynomial. We use the notation $H_R$ and $\mathcal{H}_R$ correspondingly in these two cases so that, for the knot $K$, $\mathcal{H}_R({K})=\mathcal{H}_R({U})\cdot H_R({K})$, and similarly for polynomials $P(K)$.
\end{remark}

\begin{defn}[Young tableau diagram representation \cite{KR}\cite{KS}]\label{YTDR}
A representation $R$ of the quantum group $U_q(\mathfrak{sl}_N)$ ($|q|<1$) with highest weight $\Lambda$ can be uniquely specified by a Young tableau 
\vspace{-2mm} 
\begin{equation}
\label{YoungT}
R=[R_1, R_2, \ldots, R_{N-1}]\,, \qquad 
R_1 \geq R_2 \geq \cdots \geq R_{N-1} \geq 0.
\end{equation}
If the highest weight is expressed as 
\vspace{-2mm}
\begin{equation}
\label{HW}
\Lambda_R = \sum_{i=1}^{N-1} a_i \, \omega_i \,,
\end{equation}
with fundamental weights $\{\omega_i\}_{i=1}^{N-1}$, then the Young tableau entries 
in \eqref{YoungT} 
relate to the coefficients $a_i$ in \eqref{HW} by
\vspace{-2mm}
\[
R_i = \sum_{j=i}^{N-1} a_j, \quad \text{for } i=1,\ldots,N-1.
\]
For general $N$, the conjugate representation $\bar{R}$ corresponds to the highest weight
\vspace{-2mm} 
\[
\Lambda^*_R := \sum_{i=1}^{N-1} a_{N-i} \, \omega_i,
\]

\vspace{-2mm} 
\noindent which is associated to the conjugate Young tableau.
\end{defn}

\begin{defn}[Composite representation \cite{Koike}]
The composite representation is the most general finite-dimensional irreducible highest weight representation of $\mathfrak{sl}_N$, which is associated with the Young diagram of the form
\vspace{-1mm} 
$$
(R,P)= \Big[R_1+P_1,\ldots,R_{l_R}+P_1,\underbrace{P_1,\ldots,P_1}_{N-l_{\!_R}-l_{\!_P}},
P_1-P_{_{l_{\!_P}}},P_1-P_{{l_{\!_P}-1}},\ldots,P_1-P_2\Big],
$$
where $l_P$ denotes the number of lines in the Young diagram P. This $(R,P)$ is
the first (``maximal") representation, contributing to the product $R\otimes \bar P$. It can be manifestly obtained from the tensor products (i.e., as a projector from  $R\otimes \bar P$) by formula \cite{Koike}
\vspace{-3mm}
\be
(R,P)=\sum_{Y,Y_1,Y_2}(-1)^{l_{\!_Y}}N^R_{YY_1}N^{P}_{Y^TY_2}\ Y_1\otimes\overline{Y_2}\,,
\ee
 where the superscript ``T'' denotes transposition of Young diagram.
\end{defn}

\begin{example}[Fundamental representation]
 The Young tableau diagrams for the fundamental representation $R = [1]$ and its conjugate $\bar{R} = \overline{[1]}$ of $U_q(\mathfrak{sl}_N)$ are given by
\vspace{-1mm} 
\[
[1] = {\tiny \yng(1)} \quad \text{and} \quad 
\bar{[1]} = \underbrace{[1,1,\ldots,1]}_{N-1}=
\big[1^{N-1}\big] = 
{\tiny
\left.
\begin{array}{c}
\begin{ytableau}
~ \\
~ \\
~ \\
~ \\
\none[\vdots] \\
~ \\
~ \\ 
\end{ytableau}
\end{array}
\right\}\ N-1
}\ 
.
\]
Here, $\bar{[1]}$ can be viewed as a composite representation $(\Phi,[1])$, with $\Phi$ the trivial representation. 
\end{example} 

\begin{example}[Adjoint representation]
Another simple example is the adjoint representation of $U_q(\mathfrak{sl}_N)$:
\vspace{-2mm}
\[
\mathrm{Adj}=([1],[1])= [2,\underbrace{1,\ldots,1}_{N-2}\,]=
\big[2,1^{N-2}\big] = 
{\tiny
\left.
\begin{array}{c}
\begin{ytableau}
~&~ \\
~ \\
~ \\
\none[\vdots] \\
~ \\
~ \\
\end{ytableau}
\end{array}
\right\}\ N-1
}\ .
\]
\end{example}
\begin{defn}[HOMFLY-PT polynomial for $C_2({K},t)$]
{\label{2cabledhom}}
    Let \( {K} \) be a knot, and let \( \mathcal{H}_R({K}) \) denotes the $R$-colored HOMFLY-PT polynomial of \( {K} \) colored by a Young tableaux representation \( R \) of the group \( U_{q}(\mathfrak{sl}_N)\). Given
    \vspace{-1mm} 
\[
R \otimes \bar{R} = \bigoplus_{i} Q_i\,,
\]

\vspace{-2mm} 
\noindent with each \( Q_i \in \mathrm{Rep}(U_{q}(\mathfrak{sl}_N)) \), the $R$-colored HOMFLY-PT polynomial (un-normalized) of the \emph{reverse $2$-cable knot} (which is a link) denoted \( C_2({K},t) \) is given by
\[
\mathcal{H}_{R}(C_2({K},t_\nu)) := \sum_i \mathcal{H}_{Q_i}({K})\,.
\]

\vspace{-2mm} 
\noindent In particular, for the fundamental representation (the simplest representation), i.e. $R=[1]$, we have  
\[
R \otimes \bar{R} = \Phi \oplus \mathrm{\mathrm{Adj}}\,,
\]
where $\Phi$ denotes the trivial representation and $\mathrm{\mathrm{Adj}}=[2,1^{N-2}]$ denotes the adjoint representation of $U_{q}(\mathfrak{sl}_N)$.
 Therefore, the fundamental normalized HOMFLY-PT polynomial of the reverse $2$-cable knot is
\begin{equation}{\label{HOMtwocableknot}}
P(C_{2}(K,t_\nu)) = \frac{1}{\mathcal{H}_{[1]}({U})}\big(\mathcal{H}_{\Phi}({K}) + \mathcal{H}_{\mathrm{Adj}}({K})\big),
\end{equation}
where $\mathcal{H}_{\Phi}({K}) = 1$, and $\mathcal{H}_{\mathrm{\mathrm{Adj}}}({K})$ denotes the HOMFLY-PT polynomial of knot $K$ in the adjoint representation.

\end{defn}

\begin{remark}[The vertical framing]
As any formula of representation theory origin, this formula is correct only when the polynomials are in the vertical framing \cite{AM}. Hence, throughout the paper, all HOMFLY-PT polynomials are assumed to be in the vertical framing.
\end{remark} 
\begin{remark}[Uniform HOMFLY-PT polynomial] 
The HOMFLY-PT polynomial colored by a (fixed) representation $R$ (Young diagram) is a Laurent polynomial in $q$ and $v$, which under the specialization $v=q^N$ reduces to the $\mathfrak{sl}_N$ invariant $J_R^{\mathfrak{sl}_N}(q)$. For the adjoint representation, $R_N=[2,1^{N-2}]$ as above, so that the color $R$ is specifically correlated with $v$. We define $H_{\mathrm{Adj}}(K)$ via the specialization
\[
H_{\mathrm{Adj}}(K) \;\big|_{\,v=q^N} \;=\; J_{[2,1^{N-2}]}^{\mathfrak{sl}_N}(q), 
\]
for all sufficiently large $N$ \cite{ChE}\cite{HM}. This polynomial is called \emph{uniform HOMFLY-PT polynomial} in \cite{MMM}, and the same construction applies to any composite representation.

\end{remark} 
Note that Eq. \eqref{HOMtwocableknot} provides the HOMFLY-PT polynomial for the knot~$C_2(K,t_\nu)$, consistent with the skein relation in Def.~\ref{HOM}.
Moreover, we would like to emphasize that working with Eq.~\eqref{HOMtwocableknot} is more convenient than using the skein relations.

\begin{notation}
 Throughout this article, we use the following shorthand $[x]_q$ for the $q$-number and $\{x\}$ for the antisymmetric bracket
\[
[x]_q = \frac{q^{x} - q^{-x}}{q - q^{-1}}, 
\quad 
\{x\} = x - \frac{1}{x}, 
\]
In particular, we set $z = \{q\}$.
\end{notation}

In the following subsections, we present a detailed study of the HOMFLY-PT polynomial of \emph{reverse $2$-cable torus knots}
$C_2(T(m,n),t)$ in the vertical framing, and then extend to torus links.

\subsection{HOMFLY-PT polynomial of {reverse $2$-cable torus knot} $T(m,n)$}

To evaluate the HOMFLY-PT polynomial of the knot $C_2({K},t_\nu)$, we first compute the adjoint HOMFLY-PT polynomial for the knot~$K$. To that end, we first review the quantum dimensions of Young diagram representation $R$. These quantum dimensions and their properties are associated with the knot $K$. See \cite{KS}\cite{KRT}. 

\begin{defn}[The quantum dimension of representation $R$]\label{2.1}
For the Young diagram $R$, introduce a function of two variables $q$ and $v$ labeled by $R$:
$$
D_R(q,v):=\prod_{i,j\in R}{\{v q^{j-i}\}\over \{q^{h(i,j)}\}}
$$
where $h(i,j)$ is the hook length, and the product runs over boxes of the Young diagram. Letting $R$ be a Young diagram associated with the representation of the quantum group  $U_q(\mathfrak{sl}_N)$, the quantum dimension of this representation is defined to be $D_R(q,q^N)$.
\end{defn}
\begin{example}[Quantum dimension of the adjoint]
\(R = [2,1]=\tiny{\yng(2,1)}\) \, is a Young diagram of length $|R|=3$. Then
\[
D_{[2,1]}(q,v) = \frac{v - v^{-1}}{q^{3} - q^{-3}} \cdot \frac{v q - v^{-1} q^{-1}}{q - q^{-1}} \cdot \frac{v q^{-1} - v^{-1} q}{q - q^{-1}},
\]
where the hook lengths are \(h(1,1) = 3\), \(h(1,2) = 1\), and \(h(2,1) = 1\). This representation is adjoint at $N=3$. At generic $N$, the adjoint representation is $R=\mathrm{Adj}=[2,1^{N-2}]$, i.e.,
$$
D_{[2,1^{N-2}]}={\{q^{N-1}\}\over\{q\}}\cdot{\prod_{i=1}^N{\{q^{2-i}v\}\over\{q^i\}}},
$$
so that the adjoint quantum dimension, after choosing $v=q^N$, is
\[D_{\mathrm{Adj}}=\frac{\{vq\}\{
v/q\}}{\{q\}^2}.\]
\end{example}

\begin{remark}[Properties of quantum dimension]\label{QP}
$D_R(q, v)$  is a rational function of  $q$  and $
v$, and satisfies the following properties:
\vspace{-2mm} 
\begin{align*}
\text{(i)} \quad & D_R(q, v) = D_R(q^{-1}, v^{-1}), \\
\text{(ii)} \quad & D_R(q^{-1}, v) = D_{R^{T}}(q, v), \\
\text{(iii)} \quad & D_R(q, v)=v^{|R|} q^{c(R)} A_R(q) + \cdots + (-1)^{|R|} v^{-|R|} q^{-c(R)} A_R(q), \quad 
\end{align*}
\[\text{where}~  R^T \text{ denotes the transpose of Young diagram } R, \;\; 
c(R) = \sum_{(i,j) \in R} (j - i), \; \text{and} \;\; 
A_R(q) = \prod_{(i,j) \in R} \{ q^{h(i,j)} \}^{-1}.\]
Note that Property \text{(iii)} holds when the length of the Young diagram representation \( R \), i.e. the partition length \( |R| \), is independent of the rank of the group \( N \).
\end{remark}

\begin{remark}[Properties of the unknot polynomials]
    The quantum dimension of $R$ is equal to the unknot HOMFLY-PT polynomial colored with $R$, $\mathcal{H}_R({U};q,
v)=D_R(q,
v)$, i.e., $\mathcal{H}_R({K})=D_R(q,v)\cdot H_R({K})$. Hence, the three properties in Remark \ref{QP} 
characterize the HOMFLY-PT polynomial of the unknot.
\end{remark}

\begin{remark}[Properties of HOMFLY-PT polynomial]\label{PL}
 Generally, under reversal of parameters,  the HOMFLY-PT polynomial of knot $K$ transforms as
\vspace{-2mm} 
\begin{align*}
H_R(K) &~~\xrightarrow{\;q \,\mapsto q^{-1},\, ~v \,\mapsto v^{-1}\;} ~~ H_R(!K)\,, 
\\
H_R(K) &\xrightarrow{\;q \,\mapsto q^{-1}\;} H_{R^T}(K)\,,
\end{align*}
where $H_R(!K)$ denotes the polynomial of the mirror-reflected knot $!K$.
\end{remark}

 \smallskip 

Now, we turn to the discussion of adjoint polynomials of knots, with a particular focus on torus knots $T(m,n)$. The Rosso-Jones formula for the colored HOMFLY-PT invariants of torus knots is stated in the following theorem (\cite[Theorem 5.1]{LinZheng2006}\cite[Theorem 8]{RJ}).

\begin{theorem}{\label{CHOMtorus}}
Let \({K}\) be the torus knot \(T(m,n)\), where \(m\) and \(n\) are relatively prime. Let \(R \vdash s\) be a partition. Then the reduced colored HOMFLY-PT polynomial in the vertical framing is given by:
\vspace{-1mm} 
$$
\mathcal{H}_{R}({K}) = q^{n\kappa_{R}}\cdot \sum_{\mu \vdash m~ s} c^R_{\mu} \cdot q^{-n\kappa_\mu/m} \cdot D_\mu(q,v)\,,
$$

\vspace{-3mm} 
\noindent 
where 
\vspace{-2mm} 
\begin{itemize}
  \setlength\itemsep{-1pt}
\item $v=q^N$, 
\item $\kappa_R=(\Lambda_R,\Lambda_R+2\rho)$ is the eigenvalue of the quadratic Casimir operator in the representation $\mu$ of $\mathfrak{sl}_N$,
\vspace{-2mm} 
\begin{equation}\label{Cas}
\kappa_R=2\sum_{i,j\in R}(j-i)-{s^2\over N}+s N\,,
\end{equation}

\vspace{-3mm} 
\item $\rho$ is the Weyl vector, 
\item and \(c^R_{\mu}\) are the integers determined by the Adams operation ($m$-plethystic expansion):
\vspace{-2mm} 
$$
 \widehat{Ad}_m\ S_R(x_1, x_2, \dots):=S_R(x_1^m, x_2^m, \dots)= \sum_{\mu \vdash m~s} c^\mu_{R} \cdot s_\mu(x_1, x_2, \dots).
$$
\end{itemize} 

\end{theorem}

This formula allows explicit computation of the colored HOMFLY-PT polynomial for torus knots.

\begin{remark}[Conventions]
We added a factor of $q^{-(m-1)n\kappa_R}$ as compared with the formula for $\mathcal{H}_{R}({K})$ in
\cite[Theorem 5.1]{LinZheng2006}\cite[Theorem 8]{RJ} (where the polynomial is in the topological framing) since we need the polynomials in the vertical framing in order to deal with formula (\ref{HOMtwocableknot}). This exponent $-(m-1)n\kappa_R$ in this factor is the number $(m-1)n$ of Artin generators $\sigma_i$ giving the torus knot in Definition \ref{def11}, each of them being associated with the inverse $\mathfrak{R}$-matrix (hence, the sign minus), and $q^{\kappa_R}$ emerges due to the difference in normalization of these generators in the topological and vertical framings. 
\end{remark}

\begin{remark}[Consistency]
The HOMFLY-PT polynomial in Theorem \ref{CHOMtorus} corresponds to the Reshetikhin-Turaev invariant \cite{RTmod} constructed from the inverse $R$-matrices, which are associated with the mirror knot. This corresponds to the replacement $v\to v^{-1}$, $q\to q^{-1}$ in the invariant. At the same time, the skein relation (\ref{skrel}) is associated with just the replacement  $v\to v^{-1}$. However, as soon as we are interested in the adjoint representation, and its polynomial is invariant with respect to the replace $q\to q^{-1}$, the skein relation is consistent with Theorem \ref{CHOMtorus}.
\end{remark}

Applying Theorem \ref{CHOMtorus}, we can deduce the HOMFLY-PT polynomial for { reverse $2$-cable  torus knot}. The subsequent subsections discuss the details.

\begin{example}[HOMFLY-PT polynomial of reverse $2$-cable  torus knot $T(2,n)$]
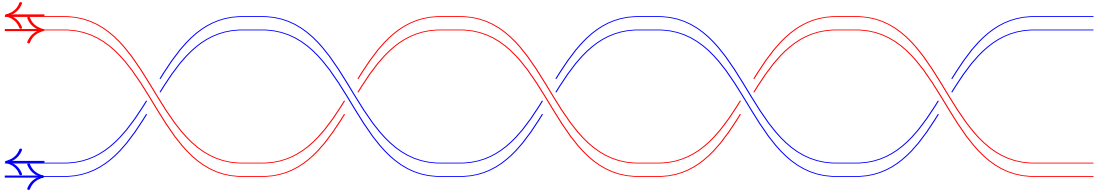
\begin{figure}[ht]
\centering
\resizebox{15cm}{2.5cm}{
\begin{tikzpicture}

\pic[rotate=90,
braid/.cd,
every strand/.style={ultra thin},
strand 1/.style={blue},
strand 2/.style={red},
strand 3/.style={black},
strand 4/.style={blue},
xshift=-1.4pt
] {braid={ s_1^{-1} s_1^{-1} s_1^{-1}s_1^{-1}s_1^{-1}}};

\pic[rotate=90,
braid/.cd,
every strand/.style={ultra thin},
strand 1/.style={blue},
strand 2/.style={red},
strand 3/.style={black},
strand 4/.style={blue},
xshift=1.2pt
] {braid={ s_1^{-1} s_1^{-1} s_1^{-1}s_1^{-1}s_1^{-1} }};

\draw[->, blue, thin] (0.0,-0.05) -- ++(0.2,0); 
\draw[<-, blue, thin] (0.0,0.05) -- ++(0.2,0);
\draw[<-, red, thin] (0.0,1.05) -- ++(0.2,0); 
\draw[->, red, thin] (0.0,.95) -- ++(0.2,0);
\end{tikzpicture}
}
\caption{Braid representation of $C_2(T_{2,5},-10)$. }
\end{figure}
The adjoint polynomials for  torus knot \(T_{2,n}\), as stated in Theorem~\ref{CHOMtorus}, can be computed explicitly \cite{MMM}: 
\vspace{-1mm}
\[
\mathcal{H}_{\mathrm{Adj}}({T_{2,n}}) = 
v^{2n} - {D}_{([1,1],[2])}(q,v)- {D}_{([2],[1,1])}(q,v) + q^{-2n} {D}_{([2],[2])}(q,v) + q^{2n} {D}_{([1,1],[1,1])}(q,v).
\]
Here
\vspace{-2mm} 
\begin{eqnarray*}
    D_{([2],[2])}(q,v)&=&\frac{\{
v\}^2\{vq^3\}\{
v/q\}}{\{q^2\}^2\{q\}^2},\qquad D_{([1,1],[1,1])}(q,v)=\frac{\{vq\}\{
v\}^2\{
v/q^3\}}{\{q^2\}^2\{q\}^2}, 
\\
D_{([1,1],[2])}&=& D_{([2],[1,1])}(q,v)=\frac{\{vq^2\}\{vq\}\{
v/q^2\}\{
v/q\}}{\{q^2\}^2\{q\}^2}.
\end{eqnarray*}
Using the result of Eq.~\eqref{HOMtwocableknot}, we then obtain the HOMFLY-PT polynomial of \(C_2(T_{2,n},-n)\)  
\[P(C_2(T(2,n),-n))= \frac{1}{D_{[1]}}\big(1+
\mathcal{H}_{\mathrm{Adj}}({T_{2,n}})\big),\]
where
\[
D_{[1]}= \frac{(-1 + 
v^2)}{(
v z)}={\{v\}\over\{q\}}
.\]
For illustration purposes, we present selected examples in Tables~\ref{t25} and \ref{t211},  with their associated geometric properties.

\begin{table}[h!]
\small 
\centering
\renewcommand{\arraystretch}{1.3}
\begin{tabular}{|>{\bfseries}c|*{7}{>{\centering\arraybackslash}m{1cm}|}}
\hline
$(z\backslash v)$ & \bf -3 & \bf -1 & \bf 1 & \bf 3 & \bf 5 & \bf 7 & \bf 9 \\
\hline
-1 & -9 & 21 & -16 & 4 & 0 & 0 & 0 \\ \hline
1  & -24 & 71 & -50 & 5 &\cellcolor{yellow!30} 1 & \cellcolor{yellow!30}1 & \cellcolor{yellow!30}1 \\ \hline
3  & -22 & 84 & -63 & 1 & 0 & 0 & 0 \\ \hline
5  & -8 & 45 & -37 & 0 & 0 & 0 & 0 \\ \hline
7  & -1 & 11 & -10 & 0 & 0 & 0 & 0 \\ \hline
9  & 0 & 1 & -1 & 0 & 0 & 0 & 0 \\ \hline
\end{tabular}
\caption{HOMFLY-PT Polynomial of $C_2(T_{2,5},-5)$ knot; the yellow box highlights a panhandle-like structure.}
\label{t25}
\end{table}

\begin{table}[h!]
\small 
\centering
\setlength{\tabcolsep}{2pt}
\begin{tabular}{|>{\bfseries}c|*{13}{>{\centering\arraybackslash}m{1cm}|}}
\hline
$(z\backslash v)$ & \bf -3 & \bf -1 & \bf 1 & \bf 3 & \bf 5 & \bf 7 & \bf 9 & \bf 11 & \bf 13 & \bf 15 & \bf 17 & \bf 19 & \bf 21 \\
\hline
-1  & -36   & 96    & -85   & 25  & 0 & 0 & 0 & 0 & 0 & 0 & 0 & 0 & 0 \\ \hline
1   & -420  & 1131  & -910  & 201 & \cellcolor{yellow!30}1 &\cellcolor{yellow!30} 1 &\cellcolor{yellow!30} 1 &\cellcolor{yellow!30} 1 & \cellcolor{yellow!30}1 & \cellcolor{yellow!30}1 &\cellcolor{yellow!30} 1 & \cellcolor{yellow!30}1 &\cellcolor{yellow!30} 1 \\ \hline
3   & -1897 & 5319  & -4032 & 610 & 0 & 0 & 0 & 0 & 0 & 0 & 0 & 0 & 0 \\ \hline
5   & -4352 & 13237 & -9805 & 920 & 0 & 0 & 0 & 0 & 0 & 0 & 0 & 0 & 0 \\ \hline
7   & -5776 & 19678 & -14673 & 771 & 0 & 0 & 0 & 0 & 0 & 0 & 0 & 0 & 0 \\ \hline
9   & -4744 & 18643 & -14275 & 376 & 0 & 0 & 0 & 0 & 0 & 0 & 0 & 0 & 0 \\ \hline
11  & -2486 & 11642 & -9262 & 106 & 0 & 0 & 0 & 0 & 0 & 0 & 0 & 0 & 0 \\ \hline
13  & -832  & 4846  & -4030 & 16  & 0 & 0 & 0 & 0 & 0 & 0 & 0 & 0 & 0 \\ \hline
15  & -172  & 1330  & -1159 & 1   & 0 & 0 & 0 & 0 & 0 & 0 & 0 & 0 & 0 \\ \hline
17  & -20   & 231   & -211  & 0   & 0 & 0 & 0 & 0 & 0 & 0 & 0 & 0 & 0 \\ \hline
19  & -1    & 23    & -22   & 0   & 0 & 0 & 0 & 0 & 0 & 0 & 0 & 0 & 0 \\ \hline
21  & 0     & 1     & -1    & 0   & 0 & 0 & 0 & 0 & 0 & 0 & 0 & 0 & 0 \\ \hline
\end{tabular}
\caption{HOMFLY-PT Polynomial of  $C_2(T_{2,11},-11)$.}
\label{t211}
\end{table}
\end{example}

\begin{example}[HOMFLY-PT polynomial of  { reverse $2$-cable  torus knot} $T(3,n)$]
\begin{figure}[ht]
\centering
\resizebox{15cm}{2.5cm}{
\begin{tikzpicture}

\pic[rotate=90,
braid/.cd,
every strand/.style={ultra thin},
strand 1/.style={blue},
strand 2/.style={red},
strand 3/.style={black},
strand 4/.style={black},
xshift=-1.4pt
] {braid={ s_2^{-1} s_1^{-1} s_2^{-1} s_1^{-1} s_2^{-1} s_1^{-1}s_2^{-1} s_1^{-1} }};

\pic[rotate=90,
braid/.cd,
every strand/.style={ultra thin},
strand 1/.style={blue},
strand 2/.style={red},
strand 3/.style={black},
strand 4/.style={blue},
xshift=1.4pt
] {braid={ s_2^{-1} s_1^{-1} s_2^{-1} s_1^{-1} s_2^{-1} s_1^{-1}s_2^{-1} s_1^{-1} }};

\draw[->, blue, thin] (0.0,-0.05) -- ++(0.2,0); 
\draw[<-, blue, thin] (0.0,0.05) -- ++(0.2,0);
\draw[<-, red, thin] (0.0,1.05) -- ++(0.2,0); 
\draw[->, red, thin] (0.0,.95) -- ++(0.2,0);
\draw[<-, black, thin] (0.0,2.05) -- ++(0.2,0); 
\draw[->, black, thin] (0.0,1.95) -- ++(0.2,0);
\end{tikzpicture}
}
\caption{Braid representation of reverse two-cable torus knot $C_2(T_{3,4},-8)$. }
\end{figure}
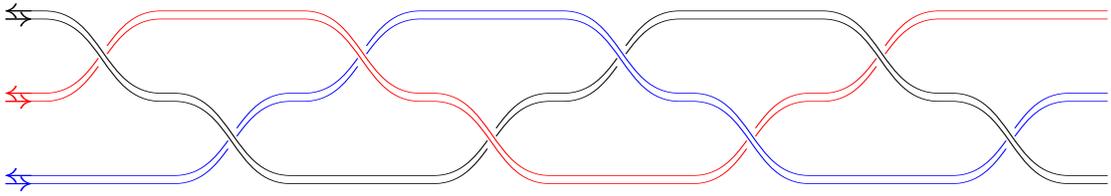

According to Theorem~\ref{CHOMtorus}, the adjoint polynomials associated with a torus knot $T_{3,n}$ admit an explicit expression \cite{MMM}: 
\begin{equation}\label{Adj0}
\mathcal{H}_{\mathrm{Adj}}({T_{3,n}}) = 
2v^{2n} + D_{X_3} + q^{-4n}D_{Y_a} + q^{4n}D_{Y_b} - q^{-2n}D_{C_a} - q^{2n}D_{C_b}\,,
\end{equation}
where 
$$
Y_a=([3],[3]),\ \ \ \ \ Y_b=([1,1,1],[1,1,1]),\ \ \ \ \ C_a=([2,1],[3])+([3],[2,1]),
$$
$$
C_b=([2,1],[1,1,1])+([1,1,1],[2,1]),\ \ \ \ \ X_3=([2,1],[2,1])+([3],[1,1,1])+([1,1,1],[3]),
$$
and
\begin{eqnarray*}
  D_{Y_a} &=& \frac{\{
v q^5\} \{
v q\}^2  \{
v\}^2  \{
v/q\}}{\{q^3\}^2 \{q^2\}^2\{q\}^2};\qquad 
D_{Y_b} = \frac{\{
v/q^5\} \ \{
v/q\}^2  \{
v\}^2  \{
v q\}}{\{q^3\}^2 \{q^2\}^2 \{q\}^2}, \\
  D_{X_3} &=& 2\frac{\{
v q^3\}  \{
v q\}  \{
v q^2\}  \{
v/q^2\} \{
v/q\} \{
v/q^3\}}{\{q^3\}^2 \{q^2\}^2 \{q\}^2} +\frac{\{
v q^3\}  \{
v q\}^2\ \{
v/q\}^2 \ \{
v/q^3\}}{\{q^3\}^2 \{q\}^4}\,,\\
D_{C_b}&=&\frac{2 \{
v\}^2 \{\frac{
v}{q^4}\} \{\frac{
v}{q^2}\} \{
v q\} \{
v q^2\}}{\{q^3\}^2\{q^2\}\{q\}^3}
; \qquad 
D_{C_a}=\frac{2 \{
v\}^2 \{
v q^4\} \{
v q^2\} \{
v/q\} \{
v/q^2\}}{\{q^3\}^2 \{q^2\} \{q\}^3}.
\end{eqnarray*}
Inserting these into Eq.~\eqref{HOMtwocableknot} yields the HOMFLY-PT polynomial for  reverse $2$-cable torus knot $T_{3,n}$
\[
P(C_2(T_{3,n},~-2n))=
\frac{1}{D_{[1]}}
\big(1+\mathcal{H}_{\mathrm{Adj}}({T_{3,n}}) \big).
\]
For clarity, we present selected examples in Tables~\ref{t34} and \ref{t37}, along with their corresponding geometric properties.

\begin{table}[h!]
\small 
\centering
\renewcommand{\arraystretch}{1.3}
\setlength{\tabcolsep}{4pt}
\begin{tabular}{|>{\bfseries}c|*{7}{>{\centering\arraybackslash}m{1cm}|}}
\hline
$(z\backslash v)$ & \bf -5 & \bf -3 & \bf -1 & \bf 1 & \bf 3 & \bf 5 & \bf 7 \\
\hline
-1 & -25  & 75   & -85  & 45  & -11 & 1 & 0 \\ \hline
1  & -100 & 350  & -408 & 206 & -44 &\cellcolor{yellow!30} 2 & \cellcolor{yellow!30}2 \\ \hline
3  & -160 & 630  & -757 & 349 & -62 & 0 & 0 \\ \hline
5  & -130 & 585  & -705 & 287 & -37 & 0 & 0 \\ \hline
7  & -56  & 308  & -363 & 121 & -10 & 0 & 0 \\ \hline
9  & -12  & 93   & -105 & 25  & -1  & 0 & 0 \\ \hline
11 & -1   & 15   & -16  & 2   & 0   & 0 & 0 \\ \hline
13 & 0    & 1    & -1   & 0   & 0   & 0 & 0 \\ \hline
\end{tabular}
\caption{HOMFLY-PT Polynomial(Panhandle polynomials) of the $C_2(T_{3,4},-8)$ knot.}
\label{t34}
\end{table}

\begin{table}[h!]
\footnotesize
\centering
\renewcommand{\arraystretch}{1.3}
\setlength{\tabcolsep}{4pt}
\begin{tabular}{|>{\bfseries}c|*{10}{>{\centering\arraybackslash}m{1.2cm}|}}
\hline
$(z\backslash v)$ & \bf -5 & \bf -3 & \bf -1 & \bf 1 & \bf 3 & \bf 5 & \bf 7 & \bf 9 & \bf 11 & \bf 13 \\
\hline
-1  & -144   & 528    & -760   & 536    & -185   & 25    & 0 & 0 & 0 & 0 \\ \hline
1   & -1584  & 5920   & -8234  & 5261   & -1459  & 102   & \cellcolor{yellow!30}2 & \cellcolor{yellow!30}2 & \cellcolor{yellow!30}2 & \cellcolor{yellow!30}2 \\ \hline
3   & -7524  & 28596  & -38772 & 22812  & -5272  & 160   & 0 & 0 & 0 & 0 \\ \hline
5   & -20328 & 79028  & -104710 & 57190 & -11310 & 130   & 0 & 0 & 0 & 0 \\ \hline
7   & -34716 & 139820 & -181104 & 91696 & -15752 & 56    & 0 & 0 & 0 & 0 \\ \hline
9   & -39492 & 167820 & -212434 & 98838 & -14744 & 12    & 0 & 0 & 0 & 0 \\ \hline
11  & -30769 & 141165 & -174526 & 73512 & -9383  & 1     & 0 & 0 & 0 & 0 \\ \hline
13  & -16610 & 84645  & -102103 & 38115 & -4047  & 0     & 0 & 0 & 0 & 0 \\ \hline
15  & -6193  & 36349  & -42715  & 13719 & -1160  & 0     & 0 & 0 & 0 & 0 \\ \hline
17  & -1562  & 11090  & -12672  & 3355  & -211   & 0     & 0 & 0 & 0 & 0 \\ \hline
19  & -254   & 2346   & -2601   & 531   & -22    & 0     & 0 & 0 & 0 & 0 \\ \hline
21  & -24    & 327    & -351    & 49    & -1     & 0     & 0 & 0 & 0 & 0 \\ \hline
23  & -1     & 27     & -28     & 2     & 0      & 0     & 0 & 0 & 0 & 0 \\ \hline
25  & 0      & 1      & -1      & 0     & 0      & 0     & 0 & 0 & 0 & 0 \\ \hline
\end{tabular}
\caption{HOMFLY-PT Polynomial(Panhandle polynomials) of the $C_2(T_{3,7},-14)$ knot.}
\label{t37}
\end{table}
\end{example}


In accordance with Theorem \ref{CHOMtorus}, to evaluate the adjoint polynomial of torus knot $T(m,n)$, first
one has to construct the $m$-plethystic expansion, and then find the quantum dimensions and the eigenvalues of the second Casimir operator.
To illustrate, we explicitly present the $4$-plethystic expansion as an example.
\begin{example}
The $4$-plethystic expansion of the adjoint Schur function generating the Adams coefficients is
\vspace{-2mm}
\[
\begin{aligned}
\widehat{Ad}_4 S_\mathrm{Adj}= 3 &+ S_{([1^4],[1^4])} -S_{([1^4],[2,1,1])} + S_{([1^4],[3,1])} - S_{([1^4],[4])} -S_{([2,1,1],[1^4])} \\
&+ S_{([2,1,1],[2,1,1])} -S_{([2,1,1],[3,1])} + S_{([2,1,1],[4])} + S_{([3,1],[1^4])} -S_{([3,1],[2,1,1])} \\
&+ S_{([3,1],[3,1])} -S_{([3,1],[4])} -S_{([4],[1^4])} + S_{([4],[2,1,1])} -S_{([4],[3,1])} + S_{([4],[4])}\,.
\end{aligned}
\]
The adjoint polynomials are
\vspace{-3mm} 
\[
\mathcal{H}_{\mathrm{Adj}}({T_{4,n}}) =
v^{2n}\bigg(3  + \sum_\mu \epsilon_\mu D_\mu(
v,q) q^{-\kappa_\mu n/4}\bigg),
\]
where $\epsilon_\mu$ is the sign factor coming from the Adams operation above. 
Explicit expressions for the quantum dimensions and the eigenvalues of the second Casimir operator in this formula are
found in \cite[Appendix A]{BM}.
\end{example}

In the generic case, we need the following result.

\begin{lemma}[Adjoint polynomials]
{\label{MainR}}
\begin{itemize}
    \item[{\bf (i)}] 
The Adams operation ${\widehat{Ad}}_m$ applied to the character of the $SL(N)$ group in the adjoint representation gives rise to $m-1$ scalar terms at $N\ge m$:
\[
\begin{aligned}
\widehat{Ad}_m\ S_\mathrm{Adj}&= (m-1)+ \sum_\mu c_\mu S_\mu(
v,q)\,,
\end{aligned}
\]

\vspace{-4mm} 
\noindent with some numeric coefficients $c_\mu$.
\vspace{-2mm} 
\item[{\bf (ii)}] 
The adjoint polynomials for \( T_{m, n} \) take the following form:
\vspace{-2mm} 
\begin{eqn}\label{Muq}
\mathcal{H}_{\mathrm{Adj}}({T_{m,n}}) =
v^{2n}\bigg(\! (m-1)  + \sum_{\mu} c_\mu D_{\mu}(q,v) q^{-\kappa_{\mu} n/m}\bigg).
\end{eqn}

\end{itemize}
\end{lemma}

\begin{proof}
First, one can use formula (\ref{Cas}) and Theorem \ref{CHOMtorus} and notice that if $R=\mathrm{Adj}$, then $\kappa_{\mathrm{Adj}}=2N$ so that the common factor is indeed $v^{2n}$.
Since the character of the $SL(N)$ group is equal to the Schur polynomial, we are interested in the scalar contribution into the Adams operation ${\widehat{Ad}}_m$, acting on the Schur function $S_{\mathrm{Adj}}(x)$, the latter being a symmetric polynomial of variables $x_i$, $i=1,\ldots,N$. These variables are eigenvalues of the group element in the fundamental representation.

Now we notice that, in this case, the adjoint Schur function is equal to
\vspace{-2mm} 
\be\label{ast1}
S_{\mathrm{Adj}}=S_{[2,1^{N-2}]}(x)=(N-1)\mathfrak{m}_{[1^N]}(x)+\mathfrak{m}_{[2,1^{N-2}]}(x)\,,
\ee
where $\mathfrak{m}_R(x)$ is the monomial symmetric function. That is, since
\vspace{-2mm}
\begin{align}
\mathfrak{m}_{[1^N]}(x)=\prod_{i=1}^Nx_i\,, \qquad 
\mathfrak{m}_{[2,1^{N-2}]}(x)&=\sum_{\sigma}x_{i_1}^2x_{i_2}x_{i_3}\ldots x_{i_{N-1}}x_{i_N}^0 \,,
\nn
\end{align}
where the sum runs over all possible permutations $\sigma$ of the set $1,2,\ldots,N$,
then\footnote{Another way to obtain this formula is to notice that the adjoint representation is obtained from the expansion (see Remark \ref{2cabledhom}) $[1]\otimes\bar{[1]}=\Phi\oplus \mathrm{Adj}$, i.e. $\Big(\sum_{i=1}^Nx_i\Big)\prod_{i=1}^Nx_i\Big(\sum_{i=1}^Nx_i^{-1}\Big)=\prod_{i=1}^Nx_i+S_{\mathrm{Adj}}$.}
\vspace{-1mm} 
\be\label{ast}
S_{[2,1^{N-2}]}(x)=\prod_{k=1}^Nx_k\sum_{i,j=1}^N{x_i\over x_j}-\prod_{i=1}^Nx_k \,.
\ee

\vspace{-1mm} 
\noindent Formula (\ref{ast1}) reflects the fact that only two Kostka numbers $K_{[2,1^{N-2}],\mu}$ (see definitions and properties in \cite[\S I.6]{Mac}) are non-zero: the diagonal one (which is always equal to 1) and $K_{[2,1^{N-2}],[1^N]}=N-1$.

The Adams operation ${\widehat{Ad}}_m$ is associated with the $m$-plethysm of this formula corresponding to the substitution $x_i\to x_i^m$. The scalar contribution to ${\widehat{Ad}}_m$ is associated with the Schur function $S_{[\underbrace{m,m,\ldots,m}_{N\ {\rm times}}]}=\prod_{i=1}^Nx_i^m$, since,  in the ${\rm SL}(N)$ case, $\prod_{i=1}^Nx_i=1$. Hence, we need to pick up from ${\widehat{Ad}}_m(S_{\mathrm{Adj}})$ the term $S_{[\underbrace{m,m,\ldots,m}_{N\ {\rm times}}]}$. To this end, we use ``another scalar product" of \cite[\S VI.9]{Mac}, the Schur functions being orthogonal with this scalar product:
\vspace{-2mm} 
\be\label{ascp}
\big<f\big|\, g\big>:&=&{1\over N!}\oint\prod_{k=1}^N{dx_k\over x_k}\prod_{i\ne j}\left(1-{x_i\over x_j}\right)f(x)g(x^{-1})\\
\big<S_R,S_Q\big>&=&\delta_{RQ}\nn \,.
\ee
Here the integration contour encircles zero, and the measure is normalized so that $\oint{dx\over x}=1$.
Then, the scalar contribution is
\vspace{-2mm}
\be
\Big<{\widehat{Ad}}_m(S_{\mathrm{Adj}}(x))\; \Big|\; S_{[\underbrace{m,m,\ldots,m}_{N\ {\rm times}}]}\Big>=\Big<S_{[2,1^{N-2}]}(x^m)\; \Big| \; \prod_{i=1}^Nx_i^m\Big>.\nn
\ee

\vspace{-2mm} 
\noindent Now, we use (\ref{ast}) and notice that
\be
\bigg<\prod_{i-1}^Nx_i^m\; \bigg| \; \prod_{i=1}^Nx_i^m\bigg>=
{1\over N!}\oint\prod_{k=1}^N {dx_k\over x_k^{m+1}}\prod_{i\ne j}\left(1-{x_i\over x_j}\right)\prod_{n=1}^Nx_n^m=1\nn
\ee
and that
\vspace{-1mm} 
\begin{align}
\bigg<\prod_{k=1}^Nx^p_k\sum_{i,j=1}^N{x_i^m\over x_j^m}\; \bigg| \; \prod_{i=1}^Nx_i^m\bigg>
&=
{1\over N!}\oint\prod_{k=1}^N {dx_k\over x_k^{m+1}}\prod_{i\ne j}\left(1-{x_i\over x_j}\right)\sum_{n,m=1}^N {x_n^m\over x_m^m}\prod_{s=1}^Nx_s^m
\nn\\
&={1\over N!}\oint\prod_{k=1}^N {dx_k\over x_k}\prod_{i\ne j}\left(1-{x_i\over x_j}\right)\sum_{n,m=1}^N {x_n^m\over x_m^m}
\nn \\
&={\rm min}(m,N)\,.
\nn
\end{align} 
Hence, we finally obtain from (\ref{ast}) that 
\be 
\Big\langle S_{[2,1^{N-2}]}(x^m) \,\Big|\, S_{[\underbrace{m,m,\ldots,m}_{N\ \text{times}}]} \Big\rangle
={\rm min}(m,N)-1\,,\nn
\ee
which is equal to $m-1$ at large enough $N$.
This completes the proof.
\end{proof}

Now we make a claim about the generic structure of the sum in Lemma \ref{MainR}.

\begin{lemma}[Structure of sums]
{\label{Main2}}
Consider the Young diagrams $\mu$ entering the sum in Lemma \ref{MainR}. Then the following hold:
\begin{itemize}
\vspace{-2mm} 
    \item[\bf {(i)}] The quantum dimensions $D_\mu(q,v)$ are all Laurent polynomials of $v$ of maximal degree $2m$ and minimal degree $-2m$. 
   
\vspace{-3mm} 
 \item[\bf {(ii)}] $\kappa_\mu-2m N$ does not depend on $N$.
\end{itemize}
\end{lemma}

\begin{proof}
The diagrams $\mu$ are all associated with composite representations \cite{Koike}. Moreover,
since the representations belong to the $m$-th power of the adjoint representation, $([1],[1])$, all of them have the form $(R,P)$ with $|R|=|P|:=\mathfrak{p}\le m$.
The quantum dimension of the composite representation is given by \cite[Formula (28)]{MM18}, which can be rewritten in the form 
\vspace{-2mm}
\be\label{qDc}
D_{(R,P)}(q,v)
=  D_{_R}\big(q,vq^{-l_{\!_P}}\big)\, D_{_P}\big(q,vq^{-l_{\!_R}}\big)\,
\prod_{i=1}^{l_{\!_R}}\prod_{j=1}^{l_{\!_P}}
{[N+R_i+P_{j}+1-i-j]_q\over[N-i-j+1]_q} \,,
\ee
where $l_R$ denotes the number of lines in the Young diagram $R$.
Since the quantum dimension can not have poles at integer $N$, the factors in the denominator are canceled with some factors in the numerator. Taking into account Definition \ref{2.1}, it follows directly that the total number of factors containing $v=q^N$, all of them being of the form $\{vq^\alpha\}$ with some $\alpha$, is $|P|+|R|=2\mathfrak{p}$. Hence, the quantum dimension of the representation $(R,P)$ is a Laurent polynomials of $v=q^N$ of maximal degree $|P|+|R|=2\mathfrak{p}$ and of minimal degree $-|P|-|R|=-2\mathfrak{p}$. 

One can also evaluate the eigenvalue of the second Casimir operator in the composite representation $(R,P)$ with $|R|=|P|=\mathfrak{p}$. It is given by
\begin{equation}\label{m}
\kappa_{(R,P)}=2N\mathfrak{p}+\kappa_R-\kappa_{P^T}\,.
\end{equation}
Now note that, in accordance with this formula and Lemma \ref{MainR} (ii), the coefficient in front of the quantum dimension $D_\mu$ enters the HOMFLY-PT polynomial with degree of $v$ equal to $2n(m-\mathfrak{p})/m$. This degree is a non-negative even number, and $m$ and $n$ are co-prime for knots, hence we have $\mathfrak{p}=m$, and
\vspace{-1mm} 
$$
\kappa_{(R,P)}-2Nm=\kappa_R-\kappa_{P^T}.
$$ 
This proves the Lemma.
\end{proof}

\subsection{Proof of the Panhandle theorem for torus knots}
\label{Sec-proof}

From Lemmas \ref{MainR} and \ref{Main2}, one can immediately prove the main theorem.

\begin{proof}[Proof of Theorem \ref{MYTH}]
Indeed, in accordance with Lemma \ref{MainR} and Lemma \ref{Main2}(ii), the term depending on $n$ can be presented in the form 
$$
(m-1) z 
v \frac{
v^{2n}-
v^{2m}}{
v^{2}-1}\,.
$$
Now, since the numerator vanishes at $v=\pm 1$ and
$P(C_2(T_{m,n},t))$ is a polynomial in $v$, such is the difference
$$
P\left(C_2\big(T_{m,n},t\big)\right)-(m-1) z 
v \frac{
v^{2n}-
v^{2m}}{
v^{2}-1}\,.
$$
Now, from Lemma \ref{Main2}(i) it follows that this
difference divided by $D_{[1]}$, call it $X$, is $[1-2m; 2m-1]_v$,
except that we need to clarify why in \eqref{Mm}
\begin{eqn}\label{66}
\md_v X=1-2m\,,\ \ \Md_v X=2m-1
\end{eqn}
instead of $\md_v X\ge 1-2m$, $\Md_v X\le 2m-1$.
That is, we show that no cancellations occur in the $\mu$-summation
of \eqref{Muq} in the highest or lowest degree $v$-term.
(This argument appears in modified form in \cite{part3}, but
is repeated here for completeness.)

We consider only $X=[P(C_2(T_{m,n},t))]_{z^{-1}}$,
and prove \eqref{66} this way.
It is well known from skein theory \cite{LickMil} that
\begin{eqn}\label{91}
[P(C_2(K,t))]_{z^{-1}}=v^{2t}(v^{-1}-v)([P(K)]_{z^0})^2\,.
\end{eqn}
Notice that for $K=T_{m,n}$ we have
\begin{eqn}\label{hyu}
\md_v P(K)=(m-1)(n-1)=2g(K)\,,
\end{eqn}
with $g(K)$ the genus.
Therefore, to prove \eqref{66}, it is enough to prove
the two claims
\begin{eqn}\label{000112}
[P(K)]_{z^{0}v^{(n\pm 1)(m-1)}}\ne 0\,.
\end{eqn}
For this, we use skein theory,
and Nakamura's observation \cite{Nakamura}.
Assume that $\bt\in B_r$ is a positive braid word (in the
Artin generators $\sg_j$).
Then for $\bt_{[i]}=\bt\sg_j^i$ and $j=1,\dots,r-1$ fixed,
and any $x,s\in \bZ$, we have for $L_i=\hat\bt_{[i]}$, and $L=L_0$,
\begin{eqn}\label{0123}
\bigl|\,[P(L_2)]_{z^{x}v^{s+2}}\,\bigr|\,=\,
\bigl|\,[P(L_1)]_{z^{x-1}v^{s+1}}\,\bigr|\,+\,
\bigl|\,[P(L)]_{z^{x}v^{s}}\,\bigr|\,.
\end{eqn}

\medskip
We write by $\chi(L)$ the maximal Euler characteristic of 
a Seifert surface of a link $L$, and 
by $l(L)$ be the number of components of $L$.
Now in \eqref{0123} we set $x=1-l(L)$, $s=1-\chi(L)$, and $r=m$.

\medskip
Let us write $\bt_1\succ \bt_2$ if
$\bt_1\in\{\ap,\ap\sg_j\}$ and $\bt_2=\ap\sg_j^2$ for a positive
braid word $\ap$. We further allow to permute letters cyclically
in the $\bt_i$. Transitively expand this relation $\succ$
to become a partial order.

Assume that $\bt$ is a positive braid so that
there is a sequence 
$(\bt=\bt_k,\bt_{k-1},\dots,\bt_0)$
with $\hat\bt_0=T_{m,n}$, and $\bt_{i-1}\succ \bt_{i}$.
Then for $L=\hat\bt$ and $l(L)$ components, and $y\ge 0$,
\begin{eqn}\label{000111}
\bigl|\,[P(T_{m,n})]_{z^0v^{(m-1)(n-1)+y}}\,\bigr|\,\ge \,
\bigl|\,[P(L)]_{z^{1-l(L)}v^{1-\chi(L)+y}}\,\bigr|\,.
\end{eqn}
It follows from \eqref{0123}.
Note that when $l(L_1)=l(L)-1$ (and not $l(L)+1$), then the first term
on the right of \eqref{0123} is zero.

We start with $\bt_0=\tau_{m,n}:=(\sg_1\dots\sg_{m-1})^n$.
This braid contains the center twist
$\Dl_m^2=(\sg_1\dots\sg_{m-1})^m$, and it is easy to see that
for every positive words $\ap,\gm$, we have $\Dl_m^2\ap\gm\succ
\Dl_m^2\gm$. This is because one can write $\Dl_m^2=\dl\sg_j$
with a positive word $\dl$ for every $j=1,\dots,m-1$. Thus 
\begin{eqn}\label{taupp}
\tau_{m,n}\succ \tau_{m,m}=\Dl_m^2\,.
\end{eqn}
Because of \eqref{000111},
this means that for \eqref{000112} it is
enough to check that the right-hand side of \eqref{000111}
is non-zero for $L$ being the closure of
$\bt_k=\Dl_m^2$, and $y=0$, $y=2(m-1)$.
But this $L$ is an $m$-component link of $m$
unknotted components.
We have 
\[
[P(L)]_{z^{1-m}}=v^{1-\chi(L)}(v^{-1}-v)^{m-1}
\]
by the known generalization of \eqref{91}. 
Therefore \eqref{000112} follows.
\end{proof}

\subsection{Extension to torus links}
\label{Sec-ExtensionToLinks}

In the case of links, the Rosso-Jones formula
(\cite[Theorem 5.1]{LinZheng2006}\cite[Theorem 8]{RJ}) reads:

\begin{theorem}{\label{CHOMtorusL}}

Consider \({L}\) be the torus link with $l$ components, \(T(m,n)\), where \(m/l\) and \(n/l\) are relatively prime integers. Let \(R \vdash s\) be a partition. Then, the reduced colored HOMFLY-PT polynomial of $L$ with all components colored by $R$ in the vertical framing is given by:
\vspace{-1mm} 
$$
\mathcal{H}_{R}({L}) = q^{n\kappa_{R}}\cdot \sum_{\mu \vdash sm} c^R_{\mu} \cdot q^{-n\kappa_\mu/m} \cdot D_\mu(q,v),
$$

\vspace{-2mm} 
\noindent where \(c^R_{\mu}\) are the integers determined by
\vspace{-2mm}
$$
 \left(\widehat{Ad}_{m/l}\ S_R(x_1, x_2, \dots)\right)^l=
 \Big(S_R\big(x_1^{m/l}, x_2^{m/l}, \dots\big)\Big)^l= \sum_{\mu \vdash m s} c^\mu_{R} \cdot s_\mu(x_1, x_2, \dots)\,.
$$
\end{theorem}

\begin{remark}
This theorem is originally formulated in \cite{LinZheng2006}\cite{RJ} for the link with components colored with $l$ generally different partitions. However, we do not need this more general case here.
\end{remark}

Examples of torus links $T(2,2n)$ and $T(3,3n)$ can be found in \cite{MMM}. Now, let us make an inspection of the generic link. Of the two essential Lemmas \ref{MainR} and \ref{Main2}, the analogue of the first, for links, reads:

\begin{lemma}[Adjoint polynomials]
{\label{MainL}}
\begin{itemize}
    \item[{\bf (i)}] 
The Adams operation ${\widehat{Ad}}_m$ applied to the character of the $SL(N)$ group in the adjoint representation to the power $l$ gives rise to $\zeta_{m,l}$ scalar terms at $N\ge m$ with 
\vspace{-2mm} 
\be\label{zeta}
\zeta_{m,l}:=\sum_{k=0}^l(-1)^{l+k}{l!\over (l-k)!}\left({m\over l}\right)^k
\ee
so that $\zeta_{m,l}\ne 0$ unless $m=l=1$.
\vspace{-3mm} 
 \item[{\bf (ii)}]  The adjoint polynomials for \( T(m, n) \) take the following form:
 \vspace{-1mm} 
\be\label{MuqL}
\mathcal{H}_{\mathrm{Adj}}({T(m,n)}) =
v^{2n}\bigg(\zeta_{m,l}  + \sum_{\mu} c_\mu^{\mathrm{Adj}} D_{\mu}(q,v) q^{-\kappa_{\mu} n/m}\bigg).
\ee
\end{itemize} 
\end{lemma}

\begin{proof}
First of all, using formula (\ref{Cas}) and Theorem \ref{CHOMtorusL} and noticing that, for $R=\mathrm{Adj}$, $\kappa_{\mathrm{Adj}}=2N$, one states that the common factor is $v^{2n}$.
In order to evaluate $\zeta_{m,l}$, we again use that the adjoint Schur function is given by (\ref{ast1}), and we are interested in the scalar contribution into degree $l$ of the Adams operation $\big(\widehat{Ad}_{m/l}\ S_\lambda(x_1, x_2, \dots)\big)^l$ as in the Rosso-Jones theorem \ref{CHOMtorusL}.

The scalar contribution to the Adams operation is associated with the Schur function $S_{[\underbrace{m,m,\ldots,m}_{N\ times}]}=\prod_{i=1}^Nx_i^m$, since,  in the $SL(N)$ case, $\prod_{i=1}^Nx_i=1$. Hence, we need to pick up from $\left(\widehat{Ad}_{m/l}\ S_\lambda(x_1, x_2, \dots)\right)^l$ the term $S_{[\underbrace{m,m,\ldots,m}_{N\ times}]}$. To this end, we again use ``another scalar product" (\ref{ascp}).
Then, the scalar contribution is
\be
\Big<\left({\widehat{Ad}}_{m/l}(S_{\mathrm{Adj}}(x))\right)^l\Big|\, S_{[\underbrace{m,m,\ldots,m}_{N\ times}]}\Big>=\Big<\Big(S_{[2,1^{N-2}]}\big(x^{m/l}\big)\Big)^l\,\Big|\,\prod_{i=1}^Nx_i^m\Big> \,.
\nn
\ee
Now, we use (\ref{ast}),
\be\label{astl}
\left(S_{[2,1^{N-2}]}\big(x^{m/l}\big)\right)^l=\prod_{j=1}^Nx_j^m\left(\sum_{i,j=1}^N{x_i^{m/l}\over x_j^{m/l}}-1\right)^l=
\prod_{j=1}^Nx_j^m\sum_{k=0}^l(-1)^{l+k}\binom{l}{k}\left(\sum_{i,j=1}^N{x_i^{m/l}\over x_j^{m/l}}\right)^k
\ee 
and notice that, at $N\ge pk/l$,
\be
\Bigg<\prod_{j=1}^Nx_j^{m}\left(\sum_{i,j=1}^N{x_i^{m/l}\over x_j^{m/l}}\right)^k\,\Bigg|\;\prod_{i=1}^Nx_i^m\Bigg>=
{1\over N!}\oint\prod_{j=1}^N {dx_j\over x_j^{m+1}}\prod_{i\ne j}\left(1-{x_i\over x_j}\right)\left(\sum_{i,j=1}^N{x_i^{m/l}\over x_j^{m/l}}\right)^k\prod_{s=1}^Nx_s^m
=k!\left({m\over l}\right)^k.
\nn
\ee
This formula, along with (\ref{astl}), implies that, at $N\ge m$,
\begin{align*}
\bigg\langle \!\left(S_{[2,1^{N-2}]}\big(x^{m/l}\big)\right)^l \,\Big|\; S_{[\underbrace{m,m,\ldots,m}_{N\ \text{times}}]}  \bigg\rangle
&=\sum_{k=0}^l(-1)^{l+k}\binom{l}{k}\Bigg<\prod_{j=1}^Nx_j^{m}\left(\sum_{i,j=1}^N{x_i^{m/l}\over x_j^{m/l}}\right)^k\Bigg|\,\prod_{i=1}^Nx_i^m\!\Bigg>
\\
&=
\sum_{k=0}^l(-1)^{l+k}{l!\over (l-k)!}\left({m\over l}\right)^k, 
\end{align*}
which exactly matches $\zeta_{m,l}$ in (\ref{zeta}).
This completes the proof.
\end{proof}

\begin{remark}[Consistency checks]
Formula (\ref{zeta}) is, indeed, equal to $m-1$ for knots, i.e., at $l=1$. Moreover, it is equal to 1 at $m=2$, $l=2$ and is equal to 2 at $m=3$, $l=3$, which perfectly matches the answers in \cite{MMM}.
\end{remark}

\begin{lemma}[Adjoint polynomials]
{\label{MainL2}}
Let $X = [[\uM; M]]_v$ denote a polynomial such that 
    \vspace{-2mm} 
\begin{eqn}\label{Mm2}
\min\deg_v X \ge \uM \quad\text{and}\quad \max\deg_v X \le M.
\end{eqn}
\begin{itemize}
    \item[\bf (i)]
The adjoint polynomial (\ref{MuqL}) can be presented in the form
\be\label{Muq2}
\mathcal{H}_{\mathrm{Adj}}\big({T(m,n)}\big) =
v^{2n}\zeta_{m,l}  + \sum_{k=1}^l\big[\big[-2mk/l; 2mk/l\big]\big]_{_v}
v^{2n(l-k)/l}.
\ee
\item[\bf (ii)] 
This adjoint polynomial normalized to unknot is
\be
{1\over \mathcal{H}_{[1]}(U)}
\mathcal{H}_{\mathrm{Adj}}\big({T(m,n)}\big) =
z{P(v)\over\{v\}}  + \sum_{k=1}^l\big[\big[-2mk/l+1; 2mk/l-1\big]\big]_{_v}
v^{2n(l-k)/l}
\ee
where $P(v)$ is a polynomial of $v$ only. 
\end{itemize}
\end{lemma}

\begin{proof}
First, let us notice that, similarly to the knot case, the diagrams $\mu$ are composite representations having form $(R,P)$ with $|R|=|P|=\mathfrak{p}\le m$.
In the link case, in accordance with Lemma \ref{MainL} and formula \eqref{m}, the diagram $\mu$ enters the HOMFLY-PT polynomial with degree of $v$ equal to $2n(m-\mathfrak{p})/m$, which is a non-negative even integer. Since, in the case of $l$-component link, both $n$ and $m$ are divided by $l$, one concludes that any $\mathfrak{p}=m k/l$, $k=1,\ldots,l$ may appear in (\ref{MuqL}). Now using Lemma \ref{Main2}(i), one concludes that formula (\ref{MuqL}) takes the form (\ref{Muq2}). This proves part (i).

By a direct analysis of formula (\ref{qDc}), one establishes that the quantum dimension of composite representation $(R,P)$ behaves at the vicinity of $v^2=1$ as
    \be
    D_{(R,P)}(q,v)\approx\sum_{k=0}d_k(q)\{v\}^k
    \ee
    and the coefficient $d_0$ does not depend on $q$.
This proves (ii). 
\end{proof}

Another point that has to be corrected in the link case is Definition \ref{2cabledhom}.
Note that for torus links the vertical framing \eqref{tnu} must be modified to
\vspace{-1mm} 
\be
\label{tforlinks}
t_\nu=\bigg(\,\underbrace{\Big(1-\frac{m}{l}\Big)\frac{n}{l},\dots,\Big(1-\frac{m}{l}\Big)\frac{n}{l}}_{\mbox{\small $l$ times}}\,\bigg)\,.
\ee

\begin{defn}[HOMFLY-PT polynomial for \emph{reverse $2$-cable link}]
{\label{2cabledhomL}}
    Let \( {L} \) be a link with $l$ components, and let \( \mathcal{H}_{\{R_\alpha\}}({L}) \) denotes the colored HOMFLY-PT polynomial of \( {L} \) colored by a set of Young tableaux representations \( R_\alpha \), $\alpha=1,\ldots,l$ of the group \( U_{q}(\mathfrak{sl}_N)\). Given
\[
R^{(\alpha)} \otimes \bar{R}^{(\alpha)} = \bigoplus_{i} Q_i^{(\alpha)}\,,
\]

\vspace{-2mm} 
\noindent with each \( Q_i^\alpha \in \mathrm{Rep}(U_{q}(\mathfrak{sl}_N)) \), the colored HOMFLY-PT polynomial (un-normalized) of the \emph{reverse $2$-cable link} is given by
\[
\mathcal{H}_{\{R_1,\ldots,R_l\}}\big(C_2({L},t_\nu)\big) := \sum_{i_1,\ldots,i_l} \mathcal{H}_{Q_{i_1}^{(1)},\ldots,Q_{i_l}^{(l)}}({L})\,.
\]
The HOMFLY-PT polynomial(normalized) of the reverse $2$-cable link $T(m,n)$ with $l$ components is   
 \be\label{Plink}
P\big(C_2({T(m,n)},t_\nu)\big)=\frac{1}{\mathcal{H}_{[1]}({U})}\sum_{k=0}^l \binom{l}{k}\mathcal{H}_{\mathrm{Adj}}\Big({T\big(mk/l,nk/l\big)}\!\Big),
\ee
where we put $\mathcal{H}_{\mathrm{Adj}}(T(0,0))=1$.
\end{defn}

\begin{example}[$2$-cable links $T(2,n)$ and $T(3,n)$]
The fundamental normalized HOMFLY-PT polynomial of the reverse $2$-cable link $T(2,n)$:
\vspace{-2mm} 
\be
P\big(C_2({T(2,n)},t_\nu)\big)=\frac{1}{\mathcal{H}_{[1]}({U})}
\Big(1 +2\underbrace{\mathcal{H}_{\mathrm{Adj}}\big({T(1,n/2)}\big)}_{D_{\mathrm{Adj}}}+ \mathcal{H}_{\mathrm{Adj}}\big({T(2,n)}\big)\Big).
\ee
Similarly, the reverse $2$-cable link $T(3,n)$ is
\vspace{-2mm} 
\be\label{3}
P\big(C_2({T(3,n)},t_\nu)\big)=\frac{1}{\mathcal{H}_{[1]}({U})}
\Big(1 +3D_{\mathrm{Adj}}+ 3\mathcal{H}_{\mathrm{Adj}}\big({T(2,2n/3)}\big)+
\mathcal{H}_{\mathrm{Adj}}\big({T(3,n)}\big)\Big).
\ee
\end{example}

The following generalizes Theorem \ref{MYTH} to the link case. In analogy to \eqref{Mm} and \eqref{Mm2}, say that $[\uM;M]]_v$ stands for
a polynomial $X$ with $\md _v X=\uM$ and $\Md_v X\le M$.

\begin{theorem}[Panhandle for links]
\label{PHT}
The HOMFLY-PT polynomial for the reverse $2$-cable $l$-component torus link $C_2(T_{m,n},t_\nu)$ has the form 
\vspace{-2mm} 
\be\label{MYTHL}
P\big(C_2(T_{m,n},t_{\nu})\big)
=\bigg[1-2m; \, 2n\ {l-1\over l}\ +\frac{2m}{l}-1\!\bigg]\bigg]_{_v} 
\; +\; 
\underbrace{\zeta_{m,l} z 
v \frac{
v^{2n}-
v^{2m}}{
v^{2}-1}}_{\rm{panhandle}}\,.
\ee
Hence, the length of the link panhandle is equal to $2(n-m)/l$.
\end{theorem}
\begin{proof}
   This follows from Lemma \ref{MainL2}(i) and formula (\ref{Plink}). In particular, the panhandle contributions to (\ref{MYTHL}) from $\mathcal{H}_{\mathrm{Adj}}({T(mk/l,nk/l)})$ give rise to $v^{2nk/l}$ terms, which are parts of the polynomial $\big[1-2m; 2n\ {l-1\over l}\ +2m/l-1\big]_{_v}$ at $k\ne l$. 
 It remains to prove that the polynomial $\big[1-2m; 2n\ {l-1\over l}\ +2m/l-1\big]\big]_{_v}$ in the Theorem is exactly this, and not just $\big[\big[1-2m; 2n\ {l-1\over l}\ +2m/l-1\big]\big]_{_v}$, i.e., that the coefficient in front of $v^{1-2m}$
   does not vanish. To this end, we use the proof of \eqref{66}, which holds but will not be 
   explicitly repeated here.
\end{proof}

\begin{remark}[Consistency check]
    This Theorem reduces to Theorem \ref{MYTH} at $l=1$.
    Note also that, when $m=n$, we still obtain $\Md_v P(C_2(T_{m,n},t_\nu))=2n-1$ because 
    the argument behind \eqref{66} remains valid.
\end{remark}
\begin{example}[Link $T(3,12)$]
As an illustration of the Theorem, the coefficients for $T(3,12)$ are
given in Table \ref{Table5}.
One can see that, starting with $z^{23}$, there is only the scalar contribution (yellow boxes); at $z^{17}$, there also emerges the contribution from $\mu=\mathrm{Adj}$ with $\mathfrak{p}=1$ (green boxes); starting with $z^{11}$, there are contributions with $\mathfrak{p}=2$ (blue boxes), and, at last, starting with $z^5$, those with $\mathfrak{p}=3$. In fact, $\mathcal{H}_{\mathrm{Adj}}({T(2,8)})$ due to formula (\ref{3}) also contributes to this table: starting with $z^{15}$, there is its panhandle contribution\footnote{For instance, the term $z~v^{15}$ comes from the sum of $-9zv^{15}$ contribution from $\mu=\mathrm{Adj}$ of $\mathcal{H}_{\mathrm{Adj}}({T(3,12)})$ in (\ref{MuqL}) --- $c^{\mathrm{Adj}}_{\mathrm{Adj}}=9$ in this
case--- of $2zv^{15}$ panhandle contribution from $\mathcal{H}_{\mathrm{Adj}}({T(3,12)})$, and of $3zv^{15}$ panhandle contribution from $\mathcal{H}_{\mathrm{Adj}}({T(2,8)})$, which totally gives $-4zv^{15}$.}, and contributions of $\mathfrak{p}=1$ and $\mathfrak{p}=2$ first emerge at $z^9$ and $z^3$ accordingly.
The sum  of numbers at every line of the table is equal to zero, which is the vanishing of the Conway
polynomial due to the link bounding a disconnected Seifert surface.

\begin{table}[h!]
\tiny
\centering
\setlength{\tabcolsep}{2pt}
\begin{tabular}{|>{\bfseries}c|*{15}{>{\centering\arraybackslash}p{1.02cm}|}}
\hline
$(z\backslash v)$ & \bf -5 & \bf -3 & \bf -1 & \bf 1 & \bf 3 & \bf 5 & \bf 7 & \bf 9 & \bf 11 & \bf 13 & \bf 15 & \bf 17 & \bf 19 & \bf 21 & \bf 23 \\
\hline
-5 & -11  & -11   & -10 & 26   & 5 & 1 & 0 & 0 & 0 & 0 & 0 & 0 & 0 & 0 & 0\\ \hline
-3  & -2138  & -3123    & -475  & 3789  & 1911 & 30 & \cellcolor{blue!30}18 & \cellcolor{blue!30}-18 & \cellcolor{blue!30}6 & 0 & 0 & 0 & 0 & 0 & 0\\
\hline
-1  & -73146,   & -112956    & -9402   & 124628  & 70234 & 522 & \cellcolor{blue!30}330 & \cellcolor{blue!30}-282 & \cellcolor{blue!30}72 & \cellcolor{green!30}0 & \cellcolor{green!30}-9 & \cellcolor{green!30}9 & 0 & 0 & 0\\ \hline
1   & -996684   & -1584078   & -95046   & 1681752 & 987298   & 5984 & \cellcolor{blue!30}2126 &\cellcolor{blue!30}-1624& \cellcolor{blue!30}272 & \cellcolor{green!30}-4 &\cellcolor{green!30}-4 & \cellcolor{green!30}2 &\cellcolor{yellow!30} 2 & \cellcolor{yellow!30}2 &\cellcolor{yellow!30} 2\\ \hline
3   & -7184691   & -11634737   & -592970   & 12111874  & 7254903 & 43737 & \cellcolor{blue!30}5676 & \cellcolor{blue!30}-4164 & \cellcolor{blue!30}372 & 0 & 0 & 0 & 0& 0 & 0 \\ \hline
5   & -31609050   & -51945059    & -2466115   & 53420915  & 32403009 & 194098 & \cellcolor{blue!30}7770 & \cellcolor{blue!30}-5790 & \cellcolor{blue!30}222 & 0 & 0 & 0 & 0& 0 & 0 \\ \hline
7   & -92427173   & -153901035    & -7147177   & 156866451  & 96049439 & 558115 & \cellcolor{blue!30}6012 & \cellcolor{blue!30}-4692 &\cellcolor{blue!30} 60 & 0 & 0 & 0 & 0& 0 & 0 \\ \hline
9   & -189895362   & -320221087    & -14912425   & 323970681  & 199954829 & 1102890 & \cellcolor{blue!30}2730 & \cellcolor{blue!30}-2262 & \cellcolor{blue!30}6 & 0 & 0 & 0 & 0& 0 & 0 \\ \hline
11  & -284749628   & -486266654    & -22982519   & 488654927  & 303785858 & 1557932 &\cellcolor{blue!30} 720 &\cellcolor{blue!30} -636 & 0 & 0 & 0 & 0 & 0& 0 & 0 \\ \hline
13  & -319999214   &-553472921    & -26692421   & 552641876   & 345909242 & 1613432 & \cellcolor{blue!30}102 & \cellcolor{blue!30}-96 & 0 & 0 & 0 & 0 & 0& 0 & 0 \\ \hline
15  & -274472117   & -480911809    & -23716007   & 477208635    & 300647119 & 1244179 &\cellcolor{blue!30} 6 &\cellcolor{blue!30} -6 & 0 & 0 & 0 & 0 & 0& 0 & 0 \\ \hline
17  & -181826796   & -322794482   & -16287176   & 318353438    & 201834960 & 720056 & 0 & 0 & 0 & 0 & 0 & 0 & 0& 0 & 0 \\ \hline
19  & -93628147   & -168439053    & -8695805   & 165120651    & 105329345 & 313009 & 0 & 0 & 0 & 0 & 0 & 0 & 0 & 0 & 0\\ \hline
21 & -37521836  & -68412618   & -3613274   & 66665888    & 42780334 & 101506 & 0 & 0 & 0 & 0 & 0 & 0 & 0 & 0 & 0\\ \hline
23 & -11653501   & -21535255    & -1163599   & 20862369    & 13465831 & 24155 & 0 & 0 & 0 & 0 & 0 & 0 & 0 & 0 & 0\\ \hline
25 & -2774868  & -5197330    & -287336   & 5005906    & 3249538 & 4090 & 0 & 0 & 0 & 0 & 0 & 0 & 0 & 0 & 0\\ \hline
27 & -496702  & -942884    & -53355   & 903021    & 589454 & 466 & 0 & 0 & 0 & 0 & 0 & 0 & 0 & 0 & 0\\ \hline
29 & -64640  & -124350   & -7206   & 118434   & 77730 & 32 & 0 & 0 & 0 & 0 & 0 & 0 & 0 & 0 & 0\\ \hline
31 & -5771  & -11249    & -668  & 10656    & 7031 & 1 & 0 & 0 & 0 & 0 & 0 & 0 & 0 & 0 & 0\\ \hline
33 & -316  & -624  & -38   & 588    & 390 & 1 & 0 & 0 & 0 & 0 & 0 & 0 & 0 & 0 & 0\\ \hline
35 & -8  & -16   & -1   & 15    & 10 & 0 & 0 & 0 & 0 & 0 & 0 & 0 & 0 & 0 & 0\\ \hline
\end{tabular}
\caption{HOMFLY-PT Polynomial of  $C_2(T_{3,12},(0,0,0))$.\label{Table5}}
\label{t211L}
\end{table}
\end{example}

\begin{remark}[Large $n$ behavior]
     In fact, at large enough $n$, contributions into $C_2(T_{m,n},t_\nu)$ are well-separated in degrees of $v$, excluding the one linear in $z$, because of the panhandle contributions. This follows from Lemma \ref{MainL2}(ii).
\end{remark}

\begin{example}[$n=42$]
\label{gaps}
    For instance, for $C_2(T_{3,42},(0,0,0))$, one can observe that
\be
C_2\big(T_{3,42},(0,0,0)\big)&=&z^{-5}\left([-5,5]_v+[-3,3]_vv^{28}\right)+\nn\\
&+&z^{-3}\left([-5,5]_v+[-3,3]_vv^{28}\right)+\nn\\
&+&z^{-1}\left([-5,5]_v+[-3,3]_vv^{28}+9(v^2-1)v^{55}\right)+\nn\\
&+&z\left([-5,5]_v+[-3,3]_vv^{28}-6v{v^{56}-v^{28}\over v^2-1}+2v{v^{84}-1\over v^2-1}\right)+\nn\\
&+&z^3\left([-5,5]_v+[-3,3]_vv^{28}\right)+\nn\\
&\ldots&\nn\\
&+&z^{35}\left([-5,5]_v+[-3,3]_vv^{28}\right).
\ee
The polynomials $[-5,5]_v$ and $[-3,3]_v$ in this formula are definitely distinct at distinct degrees of $z$.
We do not display the table for this case because of its large size.
\end{example}

\section{Geometric properties\label{S3}}

In the following subsections, we outline the geometric consequences
of Theorem \ref{MYTH}. However, we do not like
to go into details, since they require the introduction
of a series of different tools.
For this, see the sequel of papers \cite{part1}\cite{part2}\cite{part3},
with only the relevant part of these papers is summarized here.

\subsection{Arc index and Thurston-Bennequin invariant} 

An \emph{arc presentation} of
a knot or a link $L$ is an ambient isotopic image of $L$
contained in the union of finitely many half-planes,
called \emph{pages}, with a common boundary line in
such a way that each half plane contains a properly embedded single arc.
\newcommand{\arcpage}{{
                      \put(0,0){\vector(0,1){210}}
                      \put(0,10){\line(1,0){100}}
                      \put(0,190){\line(1,0){100}}
                      \put(100,10){\line(0,1){180}}
                      \put(0,10){\line(1,0){100}}}}
\newcommand{\arcends}{\put(0,25){\circle*{9}}\put(0,55){\circle*{9}}\put(0,85){\circle*{9}}
                      \put(0,115){\circle*{9}}\put(0,145){\circle*{9}}\put(0,175){\circle*{9}}}
\newcommand{\unwrap}{
\qbezier[60](-20,0)(120,0)(220,0)
\qbezier[60](-20,40)(120,40)(220,40)
\qbezier[60](-20,80)(120,80)(220,80)
\qbezier[60](-20,120)(120,120)(220,120)
\qbezier[60](-20,160)(120,160)(220,160)
\qbezier[60](-20,200)(120,200)(220,200)
\tiny
\put( -5,-20){$1$}
\put( 35,-20){$2$}
\put( 75,-20){$3$}
\put(115,-20){$4$}
\put(155,-20){$5$}
\put(195,-20){$6$}
}

\centerline{
\setlength{\unitlength}{0.010cm}
\begin{picture}(100,255)(0,-30)
\arcpage 
\put(10,-20){\scriptsize$\theta=0$}
\thicklines{\qbezier(0,85)(150,130)(0,175)}
\end{picture}\quad
\begin{picture}(100,255)(0,-30)
\arcpage 
\put(10,-20){\scriptsize$\theta=\frac{\pi}3$}
\thicklines{\qbezier(0,55)(150,85)(0,115)}
\end{picture}\quad
\begin{picture}(100,255)(0,-30)
\arcpage 
\put(10,-20){\scriptsize$\theta=\frac{2\pi}3$}
\thicklines{\qbezier(0,85)(150,115)(0,145)}
\end{picture}\quad
\begin{picture}(100,255)(0,-30)
\arcpage 
\put(10,-20){\scriptsize$\theta=\pi$}
\thicklines{\qbezier(0,25)(150,70)(0,115)}
\end{picture}\quad
\begin{picture}(100,255)(0,-30)
\arcpage 
\put(10,-20){\scriptsize$\theta=\frac{4\pi}3$}
\thicklines{\qbezier(0,55)(150,115)(0,175)}
\end{picture}\quad
\begin{picture}(100,255)(0,-30)
\arcpage 
\put(10,-20){\scriptsize$\theta=\frac{5\pi}3$}
\thicklines{\qbezier(0,25)(150,85)(0,145)}
\end{picture}\qquad
\begin{picture}(150,255)(0,-30)
\thicklines
\put(0,85){\line(0,1){90}}
\put(30,55){\line(0,1){60}}
\put(60,85){\line(0,1){60}}
\put(90,25){\line(0,1){90}}
\put(120,55){\line(0,1){120}}
\put(150,25){\line(0,1){120}}
\put(0,175){\line(1,0){120}}
\put(60,145){\line(1,0){50}} \put(150,145){\line(-1,0){20}}
\put(30,115){\line(1,0){20}} \put(90,115){\line(-1,0){20}}
\put(0,85){\line(1,0){20}} \put(60,85){\line(-1,0){20}}
\put(30,55){\line(1,0){50}} \put(120,55){\line(-1,0){20}}
\put(90,25){\line(1,0){60}}
\end{picture}
}

\smallskip 
The minimal number of pages among all arc
presentations of a link $L$ is called the
\em {arc index} of $L$ and is denoted by $a(L)$.
See \cite{MB}.

\smallskip 
A \emph{grid diagram} is a knot diagram 
which is composed of finitely many horizontal edges
and the same number of vertical edges, such that
vertical edges always cross over horizontal edges.
It is not hard to see that every knot admits a grid diagram.

The figure below explains that arc presentations and grid diagrams
correspond one-to-one. \\[0.6em]
\centerline{
\setlength{\unitlength}{0.5mm}
\begin{picture}(50,55)(-10,0)
\linethickness{0.4mm}
\put(0.7,0){\line(1,0){28.6}}
\put(29.3,30){\line(-1,0){6.3}}
\put(17,30){\line(-1,0){6.3}}
\put(10.7,10){\line(1,0){16.3}}
\put(33,10){\line(1,0){6.3}}
\put(39.3,40){\line(-1,0){18.6}}
\put(20.7,20){\line(-1,0){7.7}}
\put(7,20){\line(-1,0){6.3}}
\put(29.3,0){\line(0,1){30}}
\put(10.7,30){\line(0,-1){20}}
\put(39.3,10){\line(0,1){30}}
\put(20.7,40){\line(0,-1){20}}
\put(0.7,20){\line(0,-1){20}}
\end{picture}
\qquad \raisebox{30pt}{$\Leftrightarrow$} \qquad  
\begin{picture}(50,55)(-10,0)
{
\put(-10,10){\line(1,-1){10}}
\put(-10,10){\line(4,-1){7.5}} \put(2.5,6.875){\line(4,-1){27.5}}
\put(-10,20){\line(2,-1){7.5}} \put(2.5,13.75){\line(2,-1){7.5}}
\put(-10,20){\line(5,-1){7.5}} \qbezier(2.5,17.6)(5,17.2)(7.5,16.6)
                               \put(12.5,15.4){\line(5,-1){15}} \put(32.5,11.6){\line(5,-1){7.5}}
\put(-10,30){\line(1,-1){10}}
\put(-10,30){\line(3,-1){17.5}} \put(12.5,22.5){\line(3,-1){7.5}}
\put(-10,40){\line(2,-1){20}}
\put(-10,40){\line(4,-1){27.5}} \put(22.5,31.875){\line(4,-1){7.5}}
\put(-10,50){\line(5,-1){50}}
\put(-10,50){\line(3,-1){30}}
\qbezier[80](-10,5)(-10,29)(-10,53)
\linethickness{0.4mm}
\put(30,0){\line(0,1){30}}
\put(10,30){\line(0,-1){20}}
\put(40,10){\line(0,1){30}}
\put(20,40){\line(0,-1){20}}
\put(0,20){\line(0,-1){20}}
}
\end{picture}
}\\[1mm]

Let $\io(D)$ be the \emph{size} of a grid diagram,
the number of horizontal (or equivalently, vertical)
segments. Thus
\[
a(K)\,=\,\min\,\{\,\io(D)\,:\,\mbox{$D$ is a grid diagram of $K$}\,\}
\,.
\]
We say $D$ is a \emph{minimal} grid diagram of $K$ if
$\io(D)=a(K)$. 

\smallskip 
There is a further correspondence: when a grid diagram of $K$
is rotated by $\pi/4$, then it can be seen as a
\emph{Legendrian front diagram} of a Legendrian embedding of $K$.
See \cite{LN}\cite{Tanaka}\cite{FT}.

\smallskip 
There is an invariant of Legendrian isotopy, called
\emph{Thurston-Bennequin invariant/number}. We define it here 
equivalently in terms of a grid diagram.

\begin{defn}
The {\it Thurston-Bennequin invariant} of a grid diagram $D$,
written $TB(D)$, is defined as 
\vspace{-1mm} 
\begin{eqn}\label{92}
TB(D)=-Z(D)+w(D)\,,
\end{eqn}

\vspace{-1mm} 
\noindent where $w(D)$ is the writhe of $D$ (taken as a planar diagram 
of $K$) and $Z(D)$ the number of NW- or SE-corners of $D$.
\end{defn}

\begin{defn}
The \emph{maximal Thurston-Bennequin invariant} of a knot $K$,
written $TB(K)$ is
\vspace{-2mm} 
\[
TB(K)\,=\,\max\,\{\,TB(D)\,\vert \;\mbox{$D$ is a grid diagram of $K$}\,\}
\,.
\]
\end{defn}

\subsection{Braid index and braided surfaces}

Let $b(K)$ be the braid index of $K$,
the minimal number of strings of a braid representative of $K$.
By the MFW inequalities \cite{Morton}\cite{WF}, the writhe $w$ of an
$r$-string braid representative of $L$ satisfies
\vspace{-1mm} 
\begin{eqn}\label{MFW'}
w+r-1\ge \Md_v P(L)\ge \md_v P(L)\ge w-r+1\,,
\end{eqn}

\vspace{-2mm}
\noindent so that 
\begin{eqn}\label{MFW}
\MFW(L):=\tfrac{1}{2}\spn_v P(L)+1\le b(L)\,,
\end{eqn}
where the left-hand side is the \emph{MFW bound}
for the braid index $b(L)$.
If $\MFW(L)=b(L)$, we say that $L$ is \emph{MFW-sharp}.

\smallskip 
The Artin generators $\sg_i$ (see \eqref{Artin-gen})  are generalized to
the \emph{band generators} \cite{BKL}
\vspace{-2mm} 
\begin{eqn}\label{sij}
\sg_{i,j}\,=\,\sg_i\dots\sg_{j-2}\sg_{j-1}
\sg_{j-2}^{-1}\dots\sg_i^{-1}\,,
\end{eqn}

\vspace{-2mm} 
\noindent so that $\sg_i=\sg_{i,i+1}$.

\smallskip 
A representation of a braid $\bt$, and its closure link $L=\hat\bt$,
as a word in $\sg_{i,j}^{\pm 1}$ is called a \emph{band representation}.
A band representation of $\bt$ spans naturally a
Seifert surface $S$ of the link $L$: one
glues disks into the strands, and connects them by half-twisted
bands along the $\sg_{i,j}$. The resulting surface is called
\emph{braided Seifert surface} of $L$. 
In fact, a result of Rudolph \cite{Rudolph} (later rediscovered
independently by M.~Hirasawa) says that any Seifert surface is
of this form.

\smallskip 
A minimal genus braided Seifert surface $S$ is called a \emph{Bennequin
surface} (see \cite{BMe}). If $S$ is realizable as Bennequin
surface on $b(L)$ strings, then $S$ is called a \emph{minimal
string Bennequin surface} of $L$.
By Bennequin and Birman-Menasco \cite{BMe}, a 3-braid link $L$ has
a minimal string Bennequin surface. However, by \cite{HS}, we know that 
this is not true for 4-braid knots already.
For more work on minimal string Bennequin surfaces, see \cite{benn}.

\subsection{\Qp\ and \SP\label{SGV}}

A link is called \emph{quasipositive}
if it is the closure of a braid $\bt$ of the form
\vspace{-2mm} 
\begin{eqn}\label{qpfo}
\bt\, =\, \prod_{k=1}^\io w_k \sg_{i_k} w_k^{-1},
\end{eqn}

\vspace{-1mm} 
\noindent where $w_k$ is any braid word and $\sg_{i_k}$
is a (positive) standard Artin generator of the braid group
(see \cite{BoileauOrevkov}).
If the words $w_k \sg_{i_k} w_k^{-1}$
are of the form $\sg_{i_k,j_k}$ in \eqref{sij},
so that
\vspace{-3mm} 
\begin{eqn}\label{qpfo'}
\bt\, =\, \prod_{k=1}^\io \sg_{i_k,j_k}\,,
\end{eqn}
then they can be regarded as embedded bands.
Links which arise this way, i.e., such with
\emph{positive band presentations},
are called \emph{strongly quasipositive links}.
An overview of this topic can be found in \cite{Rudolph4}.

\begin{ques}\label{q1}(Rudolph; see \cite[Remark 8.3.3]{gener})
If $L$ is \sp, does $L$ have a \sp\ braid representative on
$b(L)$ strands?
\end{ques}

In the course of routine tabulation, the fourth author
has confirmed this for \sp\ prime knots up to 16 crossings,
for example.

\smallskip 
We write again $\chi(L)$ for the maximal Euler characteristic
of $L$, as in the proof of Theorem \ref{MYTH} in \S\ref{Sec-proof},
and $w(\bt)$ is the writhe/exponent sum of the braid $\bt$.

\begin{theorem}[Bennequin \cite{Bennequin}]
For $\hat\bt=L$, for $\bt\in B_r$, one has 
\vspace{-1mm} 
\[
-\chi(L)\ge w(\bt)-r\;.
\]
\end{theorem}

Thus, Bennequin's work implies that any
\emph{\sp\ Seifert surface}, a braided Seifert surface with only
positive bands, is a Bennequin surface: if a link $L$ has
a \sp\ Seifert surface $S$ on $r$ strands with $\io$ bands, then
$\chi(L)=\chi(S)=r-\io$. 

There is a similar version of this inequality for 
$\chi_4(L)$, the \em{(smooth) slice Euler
characteristic} of $L$,
\vspace{-1mm} 
\begin{eqn}\label{sBI}
-\chi_4(L)\ge w(\bt)-r\;.
\end{eqn}
This is sometimes called the \em{slice Bennequin inequality}.

\begin{corr}
Every \sp\ surface is a Bennequin surface.
\end{corr}

Baker-Motegi asked if
every minimal genus surface of a \sp\ link is
\sp. (See \S\ref{SBMP}.) This is an open question. 
For some work on this question, see \cite{benn}.

\smallskip 
We also note that the aforementioned examples with
Hirasawa \cite{HS} are not \sp. Thus one cannot
use their Bennequin surfaces to address Question \ref{q1}.

\begin{defn}
A link $L$ is \emph{Bennequin-sharp} if
\vspace{-2mm} 
\[
-\chi(L)\es =\es \max
\big\{\,w(\hat\bt)-r\es|\es \hat\bt=L,\ \bt\in B_r \big\}\,.
\]
Similarly, it is \emph{slice Bennequin-sharp} if
the inequality \eqref{sBI} can be made into an equality for
proper $\bt$.
\end{defn}

\begin{corr}
If $L$ is \sp, then $L$ is Bennequin-sharp.
\end{corr}

\begin{prob}[Bennequin sharpness problem; see, e.g., \cite{FLL,benn}]
\label{BSP}
A link $L$ is \sp\ if and only if $L$ is Bennequin-sharp.
\end{prob}

Returning to $C_2(K,t)$ (recall Definition \ref{2cabledhom}), set
\vspace{-1mm}
\[
\lm(K)\,:=\min \{\ t\mid C_2(K,t)\mbox{ is strongly quasipositive}\,\}\,.
\]
The following was known to Rudolph, but 
its reproof in \cite{part1} yields a fundamental insight
into the new applications. 

\begin{theorem}[Rudolph; see \cite{part1}]
If $K\ne U$ is non-trivial, then we have $\lm(K)=-TB(K)$.
Furthermore, $C_2(K,t)$ is \sp\ if and only if $t\ge \lm(K)$\,.
\end{theorem}

This relationship originates from a construction, noted
by Rudolph (see \cite[Fig.~1]{Rudolph3}) and later by Nutt
(cf. \cite[Theorem 3.1]{Nutt}), which will be
referred to as \em{grid-band construction}. See further
\cite{part1}\cite{part2}.
Note that for the unknot $\lm(U)=0$ but $-TB(U)=1$. 
We again prefer to use the symbol $\lm(K)$ instead of $TB(K)$
in the sequel. (To obtain the statements about the
maximal Thurston-Bennequin invariant, reverse signs.
The unknot can always be handled \emph{ad hoc}.)

\begin{defn}
We write {$W_{\pm}(K,t)$} for the \emph{Whitehead double} of a knot
$K$ with framing $t$ and positive/negative clasp.
\end{defn}

Rudolph's work also implies the following.

\begin{theorem}[see \cite{part2}]
We have that $W_+(K,t)$ is \sp\ if and only if $t\ge \lm(K)$\,.
And $W_-(K,t)$ is never \sp.
\end{theorem}

What is discussed in \cite{part3} is the \QP\ of these links.
In particular, we know there with S. Orevkov 
that it is \emph{not} always equivalent to their \SP.
We should emphasize here the connection between $\lm(K)$ and
$a(K)$. When $D$ is a grid diagram of a knot $K$, then one can regard
the \emph{mirror image} $!D$ as grid diagram of the mirrored knot
$!K$, by changing all crossings and rotating by $\pm \pi/2$.
With \eqref{92}, set 
\[
\lm(D)=-TB(D)=Z(D)-w(D)\,.
\]
It is straightforward to observe
\[
\lm(D)+\lm(!D)\,=\,\io(D)\,.
\]

\begin{theorem}[Dynnikov-Prasolov \cite{DP}]\label{DPT1}
If $D$ is a minimal grid diagram of $K(\ne U)$,
then $\lm(D)=\lm(K)$.
\end{theorem}

This yields the important relationship (that had been conjectured
previously)
\begin{eqn}\label{93}
\lm(K)+\lm(!K)=a(K)\,.
\end{eqn}

\subsection{Properties of torus knots}

Now we start relating the previous setting to Theorem \ref{MYTH}.
We defined two invariants \cite{part2}, which are slightly
rephrased here as follows.

\begin{defn}[\cite{part2}]
Let $K$ be a knot.
Assume $t\in\bZ$ is chosen so that
\begin{eqn}\label{0922}
\md_z P(C_2(K,t))=1-2m<0,\es
\qquad 
\Md_z P(C_2(K,t))=2n-1>0,\es
\end{eqn}
for positive integers $m,n$.
Then set
\vspace{-3mm} 
\begin{eqn}\label{defli}
\ell(K):=1+\tfrac{1}{2}\spn_v P(C_2(K,t))=m+n
\end{eqn}
and 
\vspace{-2mm} 
\begin{eqn}\label{9811}
\th(K):=t+m\,.
\end{eqn}
\end{defn}

Observe that since $C_2(K,t)$ is a 2-component
link, $\md_z P(C_2(K,t))$ and $\Md_z P(C_2(K,t))$ are always odd.
It is also easy to see from \eqref{91} that for every $K$ there
is a $t$ with \eqref{0922}.

\smallskip 
We recall that $t$ can be extracted from
$P(C_2(K,t))$, and in fact in two independent ways.
One is \eqref{91}. But an even simpler way is
\[
t=[P(1,z)]_{z}\,,
\]
using that the Conway polynomial $\nb(z)=P(1,z)$ contains
the linking number. (This was used in the much more
complicated definition of $\th(K)$ than \eqref{9811},
given in \cite{part2}, which does not assume \eqref{0922}.)
Note further the property
\[
\th(K)+\th(!K)=\ell(K)\,,
\]
in analogy to \eqref{93}.

\begin{theorem}[\cite{part2}]\label{tth}
For every non-trivial knot $K$,
\vspace{-2mm} 
\[
\ell(K)\le a(K) \quad 
\mathrm{and}  \quad 
\th(K)\le \lm(K)\,.
\]
\end{theorem}

Theorem \ref{tth} with Theorem \ref{MYTH}
implies a result of Etnyre-Honda about
the arc index of the torus knot \cite{EH},
\vspace{-2mm} 
\begin{eqn}\label{lmtr1}
a(T_{m,n})=m+n\,.
\end{eqn}
(As noted in \cite{part3}, this conclusion is not
possible from \cite{MB} when $p$ is odd.)
Theorem \ref{tth} also yields Etnyre-Honda's other result
\vspace{-1mm} 
\begin{eqn}\label{lmtr}
\lm(T_{m,n})=-mn+m+n\,,\q
\lm(!T_{m,n})=mn\,.
\end{eqn}
(We continue using `$!$' for `mirror image'.)

\smallskip 
Notice, also for later reference,
that $T_{m,n}$ has a grid diagram $D$ of size $\io(D)=m+n$,
which is the obvious generalization of
the shown here for $m=3$ and $n=5$.
\begin{eqn}\label{pqgrid}
\begin{tikzpicture}[scale=0.5, rotate=-90, line cap=round, line join=round]

\draw[black,line width=1pt] (-3,-1) -- (-3,2);
\draw[black,line width=1pt] (0,2) -- (0,5);
\draw[black,line width=1pt] (3,5) -- (3,0);

\draw[black,line width=1pt] (-2,0) -- (-2,3);
\draw[black,line width=1pt] (1,3) -- (1,6);
\draw[black,line width=1pt] (4,6) -- (4,1);

\draw[black,line width=1pt] (-1,1) -- (-1,4);
\draw[black,line width=1pt] (2,4) -- (2,-1);

\draw[white,line width=7pt] (2,0) -- (1,0);
\draw[white,line width=7pt] (3,1) -- (1,1);
\draw[white,line width=7pt] (-2,2) -- (-1,2);
\draw[white,line width=7pt] (-1,3) -- (0,3);
\draw[white,line width=7pt] (0,4) -- (1,4);
\draw[white,line width=7pt] (0.5,5) -- (1,5);

\draw[black,line width=1pt] (-3,2) -- (0,2);
\draw[black,line width=1pt] (0,5) -- (3,5);
\draw[black,line width=1pt] (3,0) -- (-2,0);

\draw[black,line width=1pt] (-2,3) -- (1,3);
\draw[black,line width=1pt] (1,6) -- (4,6);
\draw[black,line width=1pt] (4,1) -- (-1,1);

\draw[black,line width=1pt] (-1,4) -- (2,4);
\draw[black,line width=1pt] (2,-1) -- (-3,-1);

\end{tikzpicture}
\end{eqn}
We will refer to this diagram and
its planar mirror image for $!T_{m,n}$
as the \em{standard grid diagram}.

\smallskip 
Note that the reverse inequalities needed for
\eqref{lmtr1}
and \eqref{lmtr}, are
directly realized by looking at this standard grid diagram
$D$ of $T_{m,n}$
and reading off it its Thurston-Bennequin invariant
$TB(D)$ in \eqref{92}.
Note also that $\lm(T_{m,n})=1-2g(T_{m,n})$
with \eqref{hyu}, as follows alternatively from
\cite{Tanaka2}, since $T_{m,n}$ is a positive knot.

\smallskip 
The property \eqref{lmtr1} exhibits $T_{m,n}$ as
what was called in \cite{part2} an $\ell$-sharp knot.

\begin{defn}[\cite{part2}]
A knot $K$ is $\ell$-sharp if $\ell(K)=a(K)$.
\end{defn}

There, and later in \cite{part3},
 various applications of $\ell$-sharpness are studied.
For example, we also know that alternating knots are $\ell$-sharp.
(Of course, there are more: all prime knots up to 10 crossings
except $10_{132}$ are $\ell$-sharp.)

\smallskip 
Here we introduce applications for torus knots,
without detailed proofs; for those, we refer to \cite{part3}.
However, we emphasize that
these proofs sometimes directly follow from $\ell$-sharpness,
or minor additional conditions, that are ensured from
Theorem \ref{MYTH}.

\smallskip 
The first result establishes the equivalence between 
\QP\ and \SP.

\begin{corr}
Assume $K=T_{m,n}$ or $!T_{m,n}$. Then
\vspace{-2mm} 
\begin{enumerate}

\vspace{-1mm} 
\item[{\bf (i)}] each $C_2(K,t)$ is \qp\ if and only if it is \sp,

\vspace{-2mm} 
\item[{\bf (ii)}] each $W_+(K,t)$ is \qp\ if and only if it is \sp, and

\vspace{-2mm} 
\item[{\bf (iii)}] no $W_-(K,t)$ is \qp.
\end{enumerate}
\end{corr}

We can also resolve Problem \ref{BSP} for these links.

\begin{corr}
Assume $K=T_{m,n}$ or $!T_{m,n}$. Then
each $C_2(K,t)$ and $W_{\pm}(K,t)$ is \sp\ if and only if it is Bennequin-sharp.
\end{corr}

Note that for $W_-(K,t)$ this effectively, again, says that
it is not Bennequin-sharp for any $t$. The next
consequence addresses Question \ref{q1} and the
existence of a minimal string Bennequin surface.

\begin{corr}
Assume $K=T_{m,n}$ or $!T_{m,n}$. Then
each $C_2(K,t)$ and $W_{\pm}(K,t)$ has a minimal
string Bennequin surface, and if \sp, a minimal string \sp\ surface.
\end{corr}

The proof of this property of course relies on 
determining the braid index.

\begin{corr}
Assume $K=T_{m,n}$ or $!T_{m,n}$. Then
each $C_2(K,t)$ and $W_{\pm}(K,t)$ is MFW-sharp.
\end{corr}

One can write down explicit formulas.
The one for $b(C_2(K,t))$ is simplest, and is an analogue of
what Diao and Morton \cite{DM} proved for the alternating knots.
This can be obtained by applying (and substituting 
\eqref{lmtr1} and \eqref{lmtr} into) the below more general result.

\begin{prop}[\cite{part2}]
Assume $K$ is a non-trivial $\ell$-sharp knot. Then 
\begin{eqn}\label{qedt}
b(C_2(K,t))=\left\{\!\!\begin{array}{l@{\es\ \mbox{if} \es }l}
\lm(K)-t & t\le \lm(K)-a(K)\,, \\
a(K) & \lm(K)-a(K)\le t\le \lm(K)\,, \\
t-\lm(K)+a(K) & t\ge \lm(K)\,.
\end{array}
\right.
\end{eqn}
\end{prop}

\makeatletter

\long\def\vcbox#1{\setbox\@tempboxa=\hbox{#1}\parbox{\wd\@tempboxa}{\box
     \@tempboxa}}

\let\vn\varnothing

\let\sg\sigma
\def\bt{\beta}
\let\q\quad
\let\ni\noindent
\let\sS\subset
\let\ti\times
\let\pa\partial
\def\Qp{Quasipositivity}
\def\qp{quasipositive}
\def\QP{quasipositivity}
\def\SP{strong quasipositivity}
\def\sp{strongly quasipositive}
\let\es\enspace
\let\th\theta
\let\kp\kappa
\let\bc\bigcirc
\let\nb\nabla
\let\dl\delta
\let\eps\varepsilon
\let\ul\underline
\let\ol\overline
\def\ob{\overbrace}
\def\ub{\underbrace}
\def\md{\min\deg}
\def\Md{\max\deg}
\def\int{{\operator@font int}\,}
\def\ext{{\operator@font ext}\,}
\def\mcf{\min{\operator@font cf}\,}
\def\Mc{\max{\operator@font cf}\,}
\let\mc\mcf
\def\sgn{{\operator@font sgn}}
\def\MFW{\mathop {\operator@font MFW}\mathord{\!}}
\def\spn{\mathop {\operator@font span}\mathord{}}
\def\br#1{\left\lfloor#1\right\rfloor}
\def\BR#1{\left\lceil#1\right\rceil}
\let\wt\widetilde
\def\tK{\wt K}
\def\tS{\wt S}
\def\eqref#1{\mbox{(\protect\ref{#1}})}
\def\bZ{{\Bbb Z}}
\let\Dl\Delta
\let\gm\gamma
\let\lm\lambda
\def\cK{{\cal K}}
\def\cL{{\cal L}}
\def\cT{{\cal T}}
\let\So\Longrightarrow
\let\Lra\Longrightarrow
\let\Lfa\Longleftarrow
\let\tl\tilde
\let\mt\mapsto 
\let\iy\infty
\let\lra\longrightarrow
\let\sm\setminus

\let\ea\expandafter
\let\ap\alpha
\def\min{\mathop {\operator@font min}\mathord{\!}}
\def\int{\mathop {\operator@font int}\mathord{\!}}

\section{Applications to torus links\label{S6}}
\subsection{Invariants}
\subsubsection{Arc index}

{\def\mu{l}

Assume
\vspace{-1mm} 
\[
\mu=(m,n)>1\,.
\]
The torus link $T_{m,n}$ has $\mu$ components

of the knot type $T_{m/\mu,n/\mu}$.
We also maintain throughout the basic assumption
\vspace{-1mm} 
\[
n\,\ge\, m\,.
\]

\vspace{-2mm} 
\noindent
Now consider again the braid representation
of $T_{m,n}$ as an $m$-braid of $(m-1)n$ crossings.
Its reverse parallel with blackboard framing 
of each component has
$\mu$ components, each one of framing 
\be\label{48'}
\left(1-\frac{m}{\mu}\right)\frac{n}{\mu}\,.
\ee

\vspace{-1mm} 
\noindent Each component $K_i$ of $T_{m,n}$ has linking number 
\vspace{-2mm} 
\begin{eqn}\label{lnn}
lk(K_i,K_j)=\frac{m}{\mu}\cdot \frac{n}{\mu}\,
\end{eqn}
with each other component.

\begin{theorem}\label{Yh1}
The arc index for a torus link is 
$a(T_{m,n})=m+n$.
\end{theorem}

\proof 
First notice that the standard grid
diagram \eqref{pqgrid} has
its obvious generalization to torus links.
We will still refer to this diagram and
its planar mirror image
as the \em{standard grid diagram} of a torus link.

\begin{caselist}
\case $\mu<m$. Then each component $T_{m/\mu,n/\mu}$ of $T_{m,n}$ is
knotted. So
\vspace{-1mm} 
\[
a(T_{m,n})\ge \mu\cdot\left(\frac{m}{\mu}+\frac{n}{\mu}\right)=m+n\,.
\]

\case $\mu=m$, so $m\mid n$. All $p$ components of $T_{m,n}$ are
unknotted.

Let $k=\frac{n}{m}=\frac{n}{\mu}$.
Each two components of $T_{m,n}$ form a $T_{2,2k}$, and have a
$\ge 2k+2$ size in a grid diagram of $T_{m,n}$.

Take all $2$-component sublinks of $T_{m,n}$, so that
each component counts $(m-1)$-times. Then
\begin{align}
\nonumber \mbox{grid size} & \,\ge\,  \mbox{\# pairs}\cdot \frac{1}{\scbox{number of times a component counts}}\cdot 2(k+1) \\
\nonumber & \,=\,  \frac{m(m-1)}{2}\cdot \frac{1}{m-1}\cdot 2(k+1) \\
\tag*{\qed} & \,=\,  m(k+1)\,=\,n+m\,.
\end{align}

\end{caselist}

\subsubsection{Thurston-Bennequin invariants\label{STB}}

For a link $L$ of \em{numbered} components $K_i$, $i=1,\dots,\mu$,
let 
\vspace{-1mm} 
\begin{eqn}\label{LM}
M=L(t_1,\dots,t_\mu)
\end{eqn}

\vspace{-2mm} 
\noindent
be the \em{banded link} of $L$
with framing $t_i$ of the annulus around component $K_i$.
This is the obvious generalization of
$C_2(K,t)=K(t)$ for a knot $K$ (and $\mu=1$).
This construction naturally comes with a
particular pairing up of the components
of $M$ (with both components in each pair
having the same knot type), which we refer to as
a \em{banding structure}.
(There is also the suggestive generalization
of Whitehead doubles,
but for them considerable
further complications occur; see \S\ref{SWH}.)

There are a few caveats regarding the links \eqref{LM} to
put up in advance. At least for torus links $L$, the notation
is more or less unambiguous in the following sense.

\begin{lemma}\label{lmuq}
The link $T_{m,n}(t_1,\dots,t_\mu)$, regarded up to isotopy
permuting components, determines $(t_1,\dots,t_\mu)$ up to
permutation.
\end{lemma}

\proof Let $K_1,\dots,K_{2\mu}$ be the $2\mu$ torus knots
of type $T_{m/\mu,n/\mu}$ components of
\vspace{-1mm} 
\[
M=T_{m,n}(t_1,\dots,t_\mu)\,.
\]

\vspace{-2mm} 
\noindent
Then when for each fixed $K_i$, the numbers 
\[
\{\,lk(K_i,K_j)\,:\,j\ne i\,\}
\]
cancel in pairs (like $\{1,2,3,3,2\}\mt \{1,3,3\}\mt \{1\}$),
and one single number $k_i$ (in the parenthetic example $k_i=1$)
remains.
Then $(k_1,\dots,k_{2\mu})$ will give each value
an even number of times, and by replacing $(k,k)\mt k$
(as in $(1,2,2,1)\mt (1,2)$) will give $(t_1,\dots,t_{\mu})$
up to permutations. \qed

\medskip 
Note, though, that the banding structure of $M$ is
not determined uniquely up to isotopy of the
collection of bands. This is a point to keep in mind
when working with these links.

\begin{example}\label{999}
Take $M=T_{m,n}(-\frac{n}{\mu},\dots,-\frac{n}{\mu})$
for $\mu=m$. Then $M$ is the closure of the $2\mu$-string braid
of $n/\mu$ full twists, with half of its strands oriented
downward. These components are exchangeable, and thus a
fixed upward-oriented braid strand bounds a 
$-\frac{n}{\mu}$-framed annulus with every 
downward-oriented strand. This means that when the
$2\mu$ components of $M$ are numbered, there are
at least $\mu!$ different banding structures on $M$.
\end{example}

The \em{component-wise Thurston-Bennequin invariant}
of a Legendrian link $\cL$ of components $\cK_i$
is 
\vspace{-1mm} 
\begin{eqn}\label{cTB}
(TB(\cK_1),\dots,TB(\cK_\mu))\,.
\end{eqn}

\vspace{-1mm} 
\noindent
Obviously, this is the natural equivalent of
the Thurston-Bennequin invariant to study when Legendrian isotopy
of Legendrian links is considered.
However, almost everywhere simply the
extension of \eqref{92} seems treated.
For links, it yields the much coarser invariant
\vspace{-2mm} 
\begin{eqn}\label{cT1}
TB(\cL)=2lk(L)+\sum_{i=1}^{\mu} TB(\cK_i)\,,
\end{eqn}

\vspace{-1mm} 
\noindent
where $lk(L)$ is the total linking number of $L$.
It is this simplification that occurs
in \cite{DP,Tanaka,Tanaka2} for the link case.
It allows again for maximizing by setting
\vspace{-1mm} 
\begin{eqn}\label{cT2}
TB(L)=\max\,\{\,TB(\cL)\,:\,[\cL]=L\,\}\,,
\end{eqn}
where brackets denote the underlying topological
link type.

Since we like (and often have) to pay attention
to component-wise Thurston-Bennequin invariants, we will seek to avoid 
working in the framework of the above references.
In particular, our treatise below seems the first
account of extracting information from link polynomials
regarding \eqref{cTB} (Theorems \ref{th75} and \ref{th76}),
rather than merely \eqref{cT1}. Still, we will treat
the latter as well, and for obvious reasons, generalize
the notation
\vspace{-2mm} 
\begin{eqn}\label{dflm}
\lm(L)=-TB(L)\,.
\end{eqn}

\begin{remark}\label{4cu}
We caution that the vector \eqref{cTB} is far more
complex to understand than just its short-cut \eqref{cT1}.
For example, one may be able to reduce 
\[
\lm(\cK_i)=-TB(\cK_i)
\]
only at
the cost of augmenting some other $\lm(\cK_j)$. The first
example we inferred about, using Dynnikov-Prasolov
(Theorem \ref{DPT1}), is the (properly
mirrored) Whitehead link: it is exchangeable and has
odd arc index. However, we will soon see this problem
transpiring even more emphatically for (some) torus
links. In particular, maximizing \eqref{cTB} does not
seem to make much sense \em{a priori}. We will return to
this issue when we discuss corner framings in \S\ref{CAL}.
\end{remark}

\begin{theorem}\label{th75}
The tuple $(\lm(\cK_1),\dots,\lm(\cK_\mu))$ realizes the
(negated) component-wise Thurston-Bennequin invariants a
Legendrian embedding of $T_{m,n}$ 
\vspace{-3mm} 
\[
\iff\, \lm(\cK_i)\,\ge\,\left(1-\frac{m}{\mu}\right)\frac{n}{\mu}+
\frac{m}{\mu}
\]

\vspace{-1mm} 
\noindent
for all $i=1,\dots,\mu$.
\end{theorem}

We point out, that we do not know (similarly for 
Theorem \ref{th76} below) about the strength
of \eqref{cTB} for Legendrian torus links.
The peculiarities we discover with some
negative torus links, though, should serve as serious caveats
to extensions of the completeness results for knots
\cite{EH}. Of course, this study (including Maslov numbers,
etc.) goes beyond our scope and our methods here.
However, we at least compensate for the
(considerable) loss of information that occurs
when replacing \eqref{cTB} by \eqref{cT1}.

\proof
$\Lfa$. This can be seen by realizing such
an embedding from taking the grid diagram
\eqref{pqgrid} and applying component-wise
positive stabilizations. (A \em{positive
stabilization} is the addition of a short horizontal
edge at a vertical one creating a pair of
NE and SW corners; see \eqref{pbs} and \cite[(3.9)]{part2}.)

$\Lra$. Every Legendrian embedding gives a
Legendrian embedding of the component $K_i$, which is
a $T_{m/\mu,n/\mu}$.
This Legendrian embedding $\cL$
of $L$ gives rise to a grid diagram of $L$,
and thus to a \sp{} band representation of
$L(t_1,\dots,t_\mu)$ where
\[
t_i=\lm(\cK_i)
\]
with \eqref{dflm}.

Each $K_i$ appears in a sub-grid diagram, and thus
each $C_2(K_i,t_i)$ has a \sp{} band representation
\em{with bands for $K_i=U$}.
(Here we write the components $K_i$ of $L$ non-
calligraphic.)
Thus, with \eqref{dflm}, we have
\[
t_i\ge \lm(T_{m/\mu,n/\mu})=
\left(1-\frac{m}{\mu}\right)\frac{n}{\mu}+\frac{m}{\mu}\,.
\]
Note that the unknots $K_i=U$ (where $\mu=m$) are no exception:
the right-hand side\ gives $1$.
Note that $\lm(U)=0$ only due to the
existence of the empty band representation of $C_2(U,0)$,
which does not come from a grid diagram of $U$.
\qed

\smallskip 

\begin{remark}\label{Cve}
The existence of the grid diagram is essential
in going over to sublinks of $L$. Obviously
every grid diagram of $L$ yields a grid diagram
of a sublink of $L$. But the claim that
if a link $M$ is \sp{} then so is a sublink thereof
cannot be further from the truth in general; see \cite{subl}.
It is only through the grid diagram that we see this
property for (the considered sublinks of) $M=L(t_1,\dots,t_\mu)$.
\end{remark}

The consequence below is well known (also more generally,
see \cite{Tanaka2}), but (as indicated above) it is explained here
from our setting.

\begin{corr}\label{ym2}
The Thurston-Bennequin invariant of a torus link is 
$\lm(T_{m,n})=-mn+n+m$\,.
\end{corr}

\proof When $\cT_{m,n}$ is a Legendrian embedding of $T_{m,n}$,
then
\begin{eqnarray*}
\lm(\cT_{m,n}) & = & -2lk(T_{m,n})+\sum\,\lm(\cK_i) \\
& \ge & -2\cdot \left(\frac{\mu(\mu-1)}{2}\cdot \frac{m}{\mu}\cdot \frac{n}{\mu}\right)+\left(\mu\cdot 
\left(1-\frac{m}{\mu}\right)\cdot \frac{n}{\mu}+m\right) \\
& = & -\frac{mn(\mu-1)}{\mu}+n-\frac{mn}{\mu}+m \\
& = & -mn + n + m \,,
\end{eqnarray*}

\vspace{-2mm} 
\noindent and equality is realizable because for the minimal $\lm(\cK_i)$
we can have equality in the second row.
\qed

\medskip 
Now let us move to $!T_{m,n}$, the far more
interesting case.
First, we make a few remarks on $\mu=(m,n)=1$, the torus knots.
They have the negative braid representation of $(m-1)n$ crossings.
We have
\[
\lm(!T_{m,n})=mn\,.
\]
Since
\[
\lm(!\cT_{m,n})\ge mn\,=\,n+(m-1)n\,,
\]
we need in the blackboard framed $\uparrow\downarrow$
parallel (with framing $(m-1)n$) at least $n$ extra positive
full twists of each band for a \sp{} band representation.

When we go over to links $T_{m,n}$, one can still apply the
reasoning on the components $T_{m/\mu,n/\mu}$,
\em{unless} they are unknotted, i.e., $m=\mu$. This case
(as already apparent from the proof of Theorem \ref{Yh1})
will continuously require extra considerations below, thus let
us set up the following terminology.

\begin{defn}
We call a $(m,n)$-torus link \em{pure} if $m\mid n$
(i.e., $\mu=m$), and \em{non-pure} otherwise.
\end{defn}

Keeping this in mind, for $\mu=(m,n)>1$, we formulate the
complete description of component-wise Thurston-Bennequin
invariants of negative torus links.

\begin{theorem}\label{th76}
Assume 
\begin{eqn}\label{tuple}
(t_1,\dots,t_\mu)=\big(\lm(\cK_1),\dots,\lm(\cK_\mu)\big)
\end{eqn}
are the
(negated) component-wise Thurston-Bennequin invariants a Legendrian embedding
$!\cT_{m,n}$ of $!T_{m,n}$.
\vspace{-2mm} 
\begin{enumerate}
\item[\bf (i)] If $T_{m,n}$ is non-pure, then the occurring
tuples \eqref{tuple} are exactly described, for all $i=1,\dots,\mu$, by
the condition 
\vspace{-1mm} 
\begin{eqn}\label{aus}
t_i\,\ge\,\frac{m}{\mu}\cdot\frac{n}{\mu} \,.
\end{eqn}

\vspace{-4mm} 
\item[\bf (ii)] If $T_{m,n}$ is pure, then for $k=n/m$, the occurring
tuples \eqref{tuple} are exactly described by
one of the two following conditions:
\vspace{-3mm} 
\begin{itemize} 
\item[\bf (a)] Either 
\vspace{-1mm} 
\begin{eqn}\label{aut}
t_i\,\ge\,k
\end{eqn}\

\vspace{-2mm}
\noindent
for $i=1,\dots,\mu$. 

\vspace{-1mm}
\item[\bf (b)] Or $k>1$, and there is a
number $1\le u<k$ and a component $K_{i_0}$
so that
\vspace{-1mm} 
\begin{eqn}\label{aux}
t_{i_0}=u\,\q
\mbox{and}
\,\q
t_i\,\ge\,2k-u
\end{eqn}

\vspace{-2mm}
\noindent for all $i=1,\dots,\mu$ with $i\ne i_0$.
\end{itemize}
\end{enumerate}
\end{theorem}

Obviously, inequality \eqref{aut} is the special case of
\eqref{aus} when $\mu=m$. The tuples \eqref{aux}
will be called below \em{auxiliary framings}.
(The term `framings' refers to $t_i$ manifesting
themselves as the framings of the annulus
link obtained from the grid diagram of the 
Legendrian embedding using the grid-band construction
of \S\ref{SGV}~-- where of course only positive bands are used.)

We also add the following technical remarks.
Be aware that $!\cT_{m,n}$ is a notation. A 
Legendrian embedding of $!T_{m,n}$ is not simply
the mirror image of a Legendrian embedding of $T_{m,n}$.
Furthermore, we assume the mirror image operator
$!$ to bind stronger than the banding operation $L\mapsto L(\dots)$.
That is, $!L(t_1,\dots,t_\mu)$ is understood as parenthesized
like $(!L)(t_1,\dots,t_\mu)$, and not as 
$!(L(t_1,\dots,t_\mu))=(!L)(-t_1,\dots,-t_\mu)$.

\proof 
$\Lfa$ Take the planar mirror image of \eqref{pqgrid}
for a minimal grid diagram of $!T_{m,n}$, and apply
component-wise positive stabilizations
(see the proof of Theorem \ref{th75}).
This deals with realizing 
\eqref{aus} and \eqref{aut}.

To handle \eqref{aux}, first let $m=2$.
Consider the following example of 6-grid diagram of
the negative $(2,4)$-torus link,
which is shown on the right of \eqref{move}.
\begin{eqn}\label{move}
\vcbox{\begin{tikzpicture}[scale=0.7, line cap=round, line join=round]

\draw[black,line width=1pt] (-3,2) -- (-1,2); 
\draw[black,line width=1pt] (-1,4) -- (1,4); 
\draw[black,line width=1pt] (1,6) -- (-3,6);  

\draw[black,line width=1pt] (-2,3) -- (0,3); 
\draw[black,line width=1pt] (0,5) -- (2,5);  
\draw[black,line width=1pt] (2,7) -- (-2,7);  

  ==========================================================
\draw[white,line width=7pt] (0,4) -- (0,4.1); 
\draw[white,line width=7pt] (1,5) -- (1,5.1); 
\draw[white,line width=7pt] (-1,3) -- (-1,3.1); 
\draw[white,line width=7pt] (-2,6) -- (-2,6.1);


\draw[black,line width=01pt] (-1,2) -- (-1,4);
\draw[black,line width=1pt] (1,4) -- (1,6);
\draw[black,line width=1pt] (-3,2) -- (-3,6);

\draw[black,line width=1pt] (0,3) -- (0,5);
\draw[black,line width=1pt] (2,5) -- (2,7);
\draw[black,line width=1pt] (-2,3) -- (-2,7);

\end{tikzpicture}}
\q
\lra
\q
\vcbox{%
\begin{tikzpicture}[scale=0.7, line cap=round, line join=round]

\draw[black,line width=1pt] (-1,5) -- (2,5); 
\draw[black,line width=1pt] (2,7) -- (-2,7);
\draw[black,line width=1pt] (-2,4) -- (1,4);
\draw[black,line width=1pt] (1,2) -- (-1,2);

\draw[black,line width=1pt] (-3,3) -- (0,3); 
\draw[black,line width=1pt] (0,6) -- (-3,6);

\draw[white,line width=7pt] (-1,4) -- (-1,4.1); 
\draw[white,line width=7pt] (-1,3) -- (-1,3.1); 
\draw[white,line width=7pt] (0,5) -- (0,5.1); 
\draw[white,line width=7pt] (0,4) -- (0,4.1); 
\draw[white,line width=7pt] (-2,6) -- (-2,6.1);

\draw[black,line width=1pt] (-1,2) -- (-1,5);
\draw[black,line width=1pt] (2,5) -- (2,7);
\draw[black,line width=1pt] (-2,7) -- (-2,4);
\draw[black,line width=1pt] (1,4) -- (1,2);

\draw[black,line width=1pt] (0,3) -- (0,6);
\draw[black,line width=1pt] (-3,3) -- (-3,6);


\end{tikzpicture}
}
\end{eqn}
The link, $!T_{2,2}(1,3)$, obtained from the
grid-band construction
with positive bands, is shown below.
\def\rr#1{\epsfxsize.15\hsize\relax\epsffile{#1.eps}}

\def\vis#1{\hbox{\rr{#1}}}
\begin{eqn}\label{m24}
{\catcode`\_=11\relax
\vis{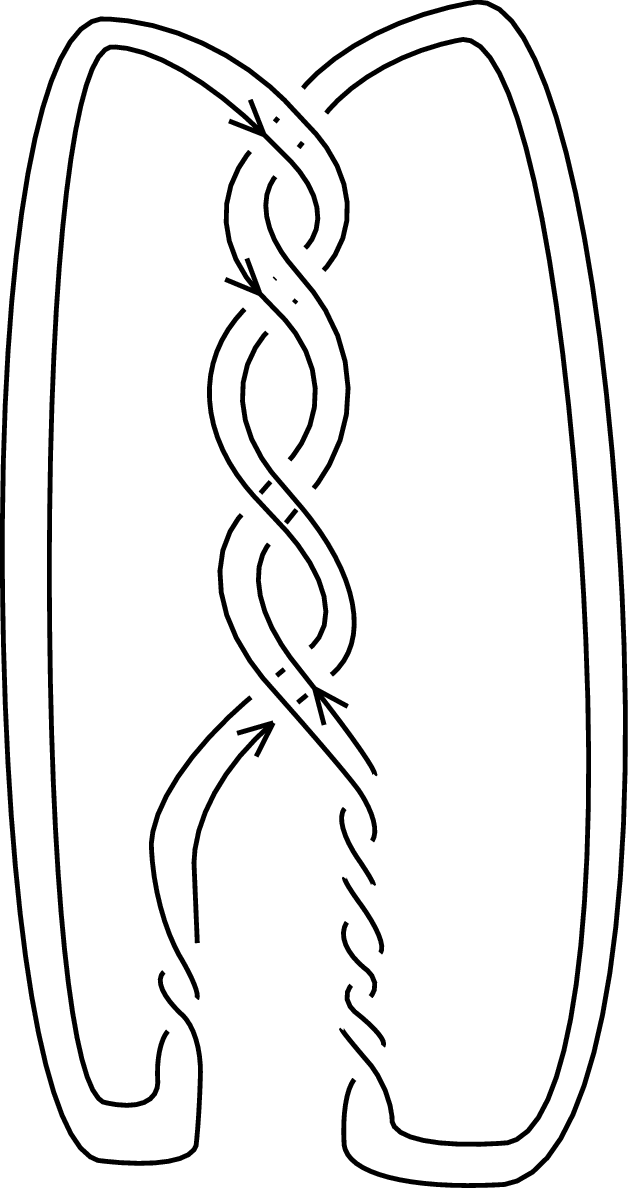}}
\end{eqn}
One can understand the grid diagram yielding \eqref{m24}
as the result of application of a move \eqref{move}
on the standard diagram,
which can be iteratively generalized. The example for
$(2,6)$-torus link explains how to
proceed for a general $(2,2k)$-torus link.
\begin{eqn}\label{m26}
\begin{tikzpicture}[scale=0.6, line cap=round, line join=round]

\draw[black,line width=1pt] (-1,5) -- (2,5);
\draw[black,line width=1pt] (2,1) -- (-2,1);
\draw[black,line width=1pt] (-2,6) -- (2,6);
\draw[black,line width=1pt] (2,8) -- (-3,8);
\draw[black,line width=1pt] (-3,4) -- (1,4);
\draw[black,line width=1pt] (1,2) -- (-1,2);

\draw[black,line width=1pt] (-4,3) -- (0,3);
\draw[black,line width=1pt] (0,7) -- (-4,7);


\draw[white,line width=5pt] (0,6) -- (0,6.1);
\draw[white,line width=5pt] (0,5) -- (0,5.1);
\draw[white,line width=5pt] (0,4) -- (0,4.1);
\draw[white,line width=5pt] (-1,4) -- (-1,4.1);
\draw[white,line width=5pt] (-1,3) -- (-1,3.1);
\draw[white,line width=5pt] (-2,4) -- (-2,4.1);
\draw[white,line width=5pt] (-2,3) -- (-2,3.1);
\draw[white,line width=5pt] (-3,7) -- (-3,7.1);

\draw[black,line width=0.8pt] (-1,2) -- (-1,5);
\draw[black,line width=0.8pt] (2,5) -- (2,1);
\draw[black,line width=0.8pt] (-2,1) -- (-2,6);
\draw[black,line width=0.8pt] (2,6) -- (2,8);
\draw[black,line width=0.8pt] (-3,8) -- (-3,4);
\draw[black,line width=0.8pt] (1,4) -- (1,2);

\draw[black,line width=0.8pt] (0,3) -- (0,7);
\draw[black,line width=0.8pt] (-4,3) -- (-4,7);

\end{tikzpicture}
\end{eqn}
To move from this to a $(m,mk)=(m,n)$-
torus link, notice that it is obtained 
from the $(2,2k)$-torus link by $m-1$-cabling
(any) one of the components with $k$ negative full twists.
And we leave it to the reader to
convince themselves that this cabling operation
can be applied (with this appropriate number of full
twists) to the component in \eqref{m26}
with self-crossings, so that each segment of
its grid is replaced by $m-1$ segments.

$\Lra$ Each Legendrian embedding $!\cT_{m,n}$ with
Thurston-Bennequin invariant $(TB(\cK_1),\dots,TB(\cK_\mu))$ gives 
rise to a grid diagram of $!T_{m,n}$ and a \sp{}
band representation of $!T_{m,n}(t_1,\dots,t_\mu)$, with
\vspace{-2mm}
\[
t_i=\lm(\cK_i)\,.
\]

\vspace{-2mm}
\noindent
In particular, this gives a \sp{}
band representation of each component 
$K_i(t_i)=C_2(K_i,t_i)$.

If $K_i\ne U$, which is when $\mu<m$, then
\vspace{-2mm}
\[
t_i\ge \lm(T_{m/\mu,n/\mu})=\frac{m}{\mu}\cdot\frac{n}{\mu}\,,
\]
and we are done. Thus we deal henceforth only with
the (far more involved) pure link case that
\[
\mbox{$\mu=m$ and $K_i=U$.}
\]

To proceed,
consider the 2-component 
sublink of $!T_{m,n}$, which is of type $M:=!T_{2,2k}$, for
$k=n/m$. The link $!T_{m,n}(t_1,\dots,t_\mu)$ must be
\sp{}, and yields a \sp{} sublink $M(t_1,t_2)$
(with the caveat of Remark \ref{Cve}), for
any choice of a two-component sublink $!T_{2,2k}$ of $!T_{m,n}$.
(Compare with the link in \eqref{m24},
which occurs for $t_1=1$, $t_2=3$ and $k=2$.)

Note that making annuli from 
the blackboard framing of the closed negative braid $\sg_1^{-2k}\in B_2$
yields $M(0,0)$.

Let $n'=2k$.
Then Theorem \ref{PHT} shows that 
\vspace{-1mm}
\[
\md_v P(M(0,0))=1-2n'=1-4k\,.
\]
(This special case, for $m=2$, can be also proved with
a similar skein-algebra tour-de-force as 
for odd $n'$ in \cite{part2}.)

By a sublink argument, using the grid diagram
(and again being aware of the Remark \ref{Cve}),
we need that
\begin{eqn}\label{99}
\mbox{when $M(t_1,t_2)$ is \sp{}, then $t_i>0$.}
\end{eqn}

Thus we can assume that 
only
positive twists ate added in the bands of $M(0,0)$
(like at the bottom of the previously recalled example in
\eqref{m24}).

Now, when a crossing is smoothed out in
$U(t_1)$, we get the split union of $U$ and
$U(t_2)$ for $t_2>0$.
This does not affect $P$-terms of negative $v$-degree.
Therefore,
\[
\md_v P(M(t_1,t_2))=1-4k+2(t_1+t_2)
\]
for $t_1+t_2< 2k$.
This means that
\begin{eqn}\label{qrt}
\md_v P(M(t_1,t_2))<0\,.
\end{eqn}
Now 
\vspace{-1mm}
\begin{eqn}\label{siz}
\chi(M(t_1,t_2))=0\,,
\end{eqn}
since
$M(t_1,t_2)$ bounds two annuli,
but not a Seifert surface with a
disk component ($M(t_1,t_2)$ has no
split unknotted component).

Thus if $M(t_1,t_2)$ is \sp{}, then
$\md_v P(M(t_1,t_2))>0$.
Hence, from \eqref{qrt},
\vspace{-1mm}
\begin{eqn}\label{hlp}
t_1+t_2\ge 2k=n'\,.
\end{eqn}

\vspace{-1mm}
\noindent
This gives then the claim:
if all $t_i\ge k$, then for each
2-component sublink of $!T_{m,n}$,
the condition \eqref{hlp} holds,
and we have \eqref{aus}.
Otherwise, let $0<t_1=:u<k$.
We have then that $t_2\ge 2k-u$ for any choice
of component $K_2$ different from $K_1$.
Thus we have \eqref{aut}.

\qed

\begin{corr}
The Thurston-Bennequin invariant of the mirror image is  
$\lm(!T_{m,n})=mn$.
\end{corr}

For this consequence of Theorem \ref{th76}
we can resort to \eqref{hlp}.

\proof 
For non-pure $!T_{m,n}$,
\vspace{-3mm}
\begin{eqnarray*}
\lm(!T_{m,n}) & \ge & \sum_i \lm(K_i) - 2\sum_{i<j} lk(K_i,K_j) \\
& = & \left(\mu\cdot \frac{m}{\mu}\cdot \frac{n}{\mu}\right) + 
2\cdot \frac{m}{\mu}\cdot \frac{n}{\mu}\cdot\frac{\mu(\mu-1)}{2}\, \\
& = & \frac{mn}{\mu}+mn-\frac{mn}{\mu}=mn\,.
\end{eqnarray*}
Here, in the first row now $lk(K_i,K_j)$ changes sign from
\eqref{lnn} due to the mirroring.

For pure $!T_{m,n}$, one can still justify 
the first pair of parentheses in the second row by averaging out
over all pairs of components $K_i,K_j$ of $!T_{m,n}$,
using \eqref{hlp}.
\qed

\medskip 
Note that this corollary can be obtained also from
Dynnikov-Prasolov and Theorem \ref{Yh1} and Corollary \ref{ym2}.
But, again, we only allude to, and do not invoke their
framework here.

\subsection{Geometric properties\label{SGT}}

\subsubsection{Arc index revisited}

With Theorem \ref{Yh1} and Theorem \ref{PHT}, we see that the $\ell$-invariant
\eqref{defli} does yield the correct value for the arc index. However,
this invariant cannot be applied straightforwardly to links.
The aim of the exposition here is to use torus links
as a starting example for generalizing the techniques
underlying the $\ell$-invariant to links. This will also 
suggest ways how to extend the $\ell$-invariant itself
(\S\ref{CAL}).

\smallskip 
We continue using the banding of the blackboard framing
of $T_{m,n}$ as a positive $m$-braid of $(m-1)n$
crossings. One has $\mu$ components, each with
framing 
\begin{eqn}\label{dll}
\dl:=\left(1-\frac{m}{\mu}\right)\frac{n}{\mu}\,.
\end{eqn}
Since the shift of framing by $\dl$ will be so
common below, we introduce the following extra notation
to save writing.

\begin{defn}
Let us say $(e_1,\dots,e_\mu)$ is the \em{corrected framing}
if $(\dl+e_1,\dots,\dl+e_\mu)$ is the real framing.
Write 
\[
\tl T_{m,n}(e_i)_i=T_{m,n}(\dl + e_i)_i
\]

\vspace{-2mm} 
\noindent
for the
banding link of $T_{m,n}$ with corrected framing $(e_i)_i$.
\end{defn}

Then Theorem \ref{PHT} implies
that 
\vspace{-3mm} 
\[
\Md_v P\big(\tl T_{m,n}(0,\dots,0)\big)=2n-1\,.
\]

\begin{corr}\label{lsrc}
The maximum degree of the torus link with corrected framing is given by 
\vspace{-2mm} 
\[
\Md_v P\big(\tl T_{m,n}(e_1,\dots,e_\mu)\big)=2n-1+2\sum_{i=1}^\mu e_i
\]

\vspace{-2mm} 
\noindent when all $e_i\ge 0$.
\end{corr}

\proof 

We perform induction over 
\[
\bigg(\mu,\; \sum_{i=1}^\mu e_i\bigg)\,,
\]
using the skein relation \eqref{skrel} applied at a positive crossing
of $\tl K_i(e_i)$,
\vspace{-2mm} 
\begin{eqn}\label{922}
P\big(\tl T_{m,n}(e_1,\dots,e_\mu)\big)=vz P_0+v^2P_-\,.
\end{eqn}

\vspace{-1mm} 
\noindent
Note that when
deleting a component $\tl K_i(e_i)$ of $\tl T_{m,n}(e_1,\dots,e_\mu)$,
one obtains 
\vspace{-2mm} 
\begin{eqn}\label{vty}
\tl T_{m(\mu-1)/\mu,n(\mu-1)/\mu}(e_1,\dots,e_{i-1},e_{i+1},\dots,e_\mu)\,.
\end{eqn}

\vspace{-1mm} 
\noindent
Thus $P_0$ in \eqref{922} is the polynomial of the split union of
\eqref{vty} and an unknot $U$, and $P_-$ is the 
polynomial of 
\vspace{-2mm} 
\[
\tl T_{m,n}(e_1,\dots,e_{i-1},e_i-1,e_{i+1},\dots,e_\mu)\,.
\]
Using that 
\begin{eqn}\label{qqq}
n(\mu-1)/\mu<n
\end{eqn}
and induction on the number of components $\mu$ for the $P_0$ term, 
and induction over $\sum_{i=1}^\mu e_i$
for the $P_-$ term,
we see that the leading $v$-degree term of
$v^2P_-$ in \eqref{922} is always of higher $v$-degree
than the leading $v$-degree term of $vz P_0$.

\qed

\smallskip 
\begin{corr}
The arc index is given by $a(T_{m,n})=m+n$.
\end{corr}

The point of repeating this result is that now we emphasize
a technique that leads later to extending the $\ell$-invariant
(\S\ref{CAL}),
and builds more on the HOMFLY-PT polynomial's own capacity,
than on sublink arguments.

\proof
Assume there is a grid diagram of $T_{m,n}$.
This gives, via the grid-band construction, a \sp{} banding
link $M=\tl T_{m,n}(e_i)_i$, where we need 
\begin{eqn}\label{0022}
e_i\ge \frac{m}{\mu}
\end{eqn}
because of \SP{} of components.
(We will deal with the auxiliary framings extra below.)
We have, by applying \eqref{66} as in the
proof of Theorem \ref{PHT}, and degree shift from linking numbers,
\vspace{-1mm} 

\[
\md_v [P(M)]_{z^{1-2\mu}}=1\,,
\]

\vspace{-1mm} 
\noindent
when all $e_i=\frac{m}{\mu}$.
Thus, by (further) switch of linking numbers
\vspace{-3mm} 
\[
\md_v [P(M)]_{z^{1-2\mu}}=1+2\sum_{i=1}^\mu \left(e_i-\frac{m}{\mu}\right)
\]
for $e_i$ in \eqref{0022}.
By Corollary \ref{lsrc},
\vspace{-2mm} 
\begin{eqn}\label{seven}
\Md_v P(M)=2n-1 + 2\sum_{i=1}^\mu e_i\,.
\end{eqn}
Then
\vspace{-3mm} 
\begin{eqnarray*}
2\MFW(M)-2 & \ge & \Md_vP(M)-\md_v [P(M)]_{z^{1-2\mu}} \\
& = & 2n-1 +2m -1 \,=\,2m+2n-2\,.
\end{eqnarray*}
Now note that for the auxiliary framings (of pure links)
the argument applies as well.
Thus
\[
\MFW(M)\ge m+n\,.
\]
Hence, every \sp{} banding link of $T_{m,n}$ has braid
index at least $m+n$. This gives $a(T_{m,n})\ge m+n$,
with the reverse inequality being obvious from \eqref{pqgrid}.
\qed

\subsubsection{Minimal string \sp{} surfaces\label{MSS}}

\begin{corr}
Let $M$ be a \sp{} banding of $T_{m,n}$ with corrected framing
$(e_1,\dots,e_\mu)$ for \eqref{0022}.
Then the braid index is 
\vspace{-2mm} 
\[
b(M)=n+\sum_{i=1}^{\mu} e_i\,.
\]
\end{corr}

\proof We use \eqref{seven} and
\begin{eqn}\label{eight}
``\md_v P(M)=1". 
\end{eqn}
We do not know whether this is always true,
but it can be assumed true for braid index purposes,
because of Bennequin's inequality and
\[
\chi(M)=0
\]
(with the same argument as for \eqref{siz}).
Thus, even if \eqref{eight} is not true, we can always 
assume it true for estimating $b(M)$ using $\MFW(M)$
from \eqref{MFW}.
With this justification, \eqref{seven} and 
\eqref{eight} gives 
\vspace{-3mm} 

\begin{eqn}\label{nine}
b(M)\ge ``\MFW(M)" \,=\, \tfrac{1}{2}\big(\Md_v P(M)-``\md_v P(M)"\big)+1\,=\,n+\sum_{i=1}^\mu e_i\,.
\end{eqn}
The reverse inequality holds, because 
one can find a grid diagram of $T_{m,n}$
with $\lm(K_i)=\dl+e_i$ (see \eqref{dll})
by using component-wise positive stabilization
of the standard grid diagram \eqref{pqgrid}.
This process yields a grid diagram of the size being
the right-hand side\ of \eqref{nine}. \qed

\begin{corr}
If $M$ is a \sp{} banding link of $T_{m,n}$,
then $M$ has a minimal string \sp{} surface. 
\end{corr}

The problem to determine the braid index
(and minimal string Bennequin surface) of an
arbitrary banded link of $T_{m,n}$ remains more
complex. We suspect that, if Diao-Morton type of
formulas \eqref{qedt} hold
for braid indices of bandings of (even only torus) links,
they will be very involved,
and it is extremely difficult to control degeneracies
(even in our special cases).

\medskip 
Now consider the mirror $!T_{m,n}$. Every component is $!T_{m/\mu,n/\mu}$,
and the minimal \sp{} annulus framing is $\frac{mn}{\mu^2}$ unless
$m=\mu$ (where we have the constructions \eqref{move},\eqref{m26}
of auxiliary framings).
Recall Theorem \ref{th76}.

\begin{theorem}\label{prt}
Assume $M=!T_{m,n}(t_1,\dots,t_\mu)$ is a \sp{} banding
of $!T_{m,n}$. Then $M$ has a minimal string \sp{} surface.
\end{theorem}

While again one can use sublinks, the argument here relies more
on the capacity of $P$ itself, assisted by 
Bennequin's inequality through \eqref{eight}. We see again that they
are sufficient to determine the braid index of $M$.

\proof We observe again from Theorem \ref{PHT}, by taking
the mirror image (and, modulo coefficient signs,
$v$ to $v^{-1}$), that
\[
!T_{m,n}(\kp,\dots,\kp),\quad \,\kp=\left(\frac{m}{\mu}-1\right)\cdot\frac{n}{\mu}
\]
(note the change of sign in \eqref{48'}) had 
\[
\Md_v [P]_{z^{1-2\mu}}=2m-1\,.
\]
It follows, by just changing linking numbers, that
\[
M_0=!T_{m,n}\left(\frac{mn}{\mu^2}\,,\dots,\frac{mn}{\mu^2}\right)
\]
has
\[
\Md_v [P(M_0)]_{z^{1-2\mu}}=2(m+n)-1\,.
\]
Then
\[
\Md_v \left[P\left(!T_{m,n}\left(t_1,\dots,t_\mu\right)\right)\right]_{z^{1-2\mu}}=
2(m+n)-1+2\sum_{i=1}^{\mu} \left(t_i-\frac{mn}{\mu^2}\right)
\]
for $t_i\ge \frac{mn}{\mu^2}$, by linking number reasons.

With \eqref{eight}, we have
\begin{eqn}\label{Mrq}
\,b(M)\,\ge\,``\MFW(M)''\,=\,m+n+
\sum_{i=1}^{\mu} \left(t_i-\frac{mn}{\mu^2}\right)\,.
\end{eqn}
Conversely, $M$ does have a braid (band)
representative on (right-hand side of \eqref{Mrq})
strands, by the grid-band construction applied on the
appropriate positive stabilization of the 
mirrored diagram \eqref{pqgrid}.

Thus the braid index of $M$ is given by 
right-hand side of \eqref{Mrq}, and $M$ has a \sp{}
braid representative on that many strands.

The auxiliary framings (for the pure torus links)
can be handled similarly, and we leave them
to the reader. \qed

\medskip 
The idea in this proof is that one can adapt the
definition of $\ell$ as an arc index bound
of a $\mu$-component link $L$,
when one replaces $\Md_v P$ of \sp{} banded links of $L$
by $\Md_v [P]_{z^{1-2\mu}}$.
This is an alternative path to considering (banded links of)
sublinks of $L$. (See further \S\ref{CAL}.)

\medskip 

As for $T_{m,n}$, it is (equally and) too complicated to
consider the braid index and Bennequin surface of an arbitrary 
banded link of $!T_{m,n}$.

\subsubsection{Bennequin sharpness}

Here we can extend the Bennquin sharpness results
and, partially, the equivalence of \QP{} and \SP{}.

\begin{theorem}\label{TTK}
Let $L=T_{m,n}$ or $!T_{m,n}$. Assume $M=L(t_1,\dots,t_\mu)$
is Bennequin-sharp. Then it is \sp{}.
\end{theorem}

\def\labelenumi{\arabic{enumi})}

\proof 
The idea in both cases is to see that
Bennequin sharpness of $M$ restricts $t_i$ to be as
in Theorems \ref{th75} and \ref{th76}.
The theorems then say (via the grid-band construction)
that $M$ is \sp{}.

Let us write below $w(\bt)$ for the writhe
(exponent sum) of a braid $\bt\in B_r$ and $r=r(\bt)$ for
its string number.

\begin{enumerate}
\item[\bf (i)] $L=!T_{m,n}$. If $L=L_1\cup L_2$,
where $L_i$ are $\mu_i$-component sublinks of $L$,
then 
\[
M=M_0=L(t_1,\dots,t_\mu)=M_1\cup M_2
\]
with
\[
M_1=L_1(t_1,\dots,t_{\mu_1})\,,
\quad
M_2=L_2(t_{\mu_1+1},\dots,t_\mu)\,.
\]
Again, $M_i$ have no split unknots, so
$\chi(M_i)=0$ (with the same argument as
for \eqref{siz}).
Also
\[
lk(M_1,M_2)=0
\]
in $M$. We will argue, in \eqref{66n},
that a braid $\bt$ with $\hat\bt=M$ making $M$
Bennequin sharp, splits into subbraids $\bt_i$ with
$\hat\bt_i=M_i$ making $M_i$ Bennequin sharp.

When $\mu<m$, then argue with $1$-component 
sublinks of $M$, analogously to the case $\mu=m$ that
just follows.

For the case $\mu=m$, use 2-component sublinks.
As below, we see that a 2-component
banded sublink 
\[
Y=!T_{2,n'}(t_1,t_2)
\]
for ($n'=2n/\mu$) is Bennequin sharp implies that 
\begin{eqn}\label{qw}
t_1+t_2\ge n'\,,\q t_i>0\,.
\end{eqn}

Namely, since $\chi(Y)=0$, assume that 
\[
\mbox{$Y$ has a braid representative $\bt$ with $w(\bt)=r(\bt)$}.
\]
Thus $\bt$ has subbraids $\bt_{1,2}$ with $\hat\bt_i=C_2(U,t_i)=U(t_i)$.

We claim, for $i=1,2$,
\begin{eqn}\label{66n}
w(\bt_i)=r(\bt_i)\,.
\end{eqn}
Since $lk(U(t_1),U(t_2))=0$, we have 
\[
w(\bt_1)+w(\bt_2)=w(\bt)\,,
\]
and obviously also
\[
r(\bt_1)+r(\bt_2)=r(\bt)=w(\bt)\,.
\]
If $w(\bt_i)>r(\bt_i)$ for some $i=1,2$, then
we would have a contradiction to Bennequin's inequality for $U(t_i)$,
since $-\chi(U(t_i))\le 0$ (even if $t_i=0$, the argument works).
Thus \eqref{66n} holds.

To obtain \eqref{qw}, we checked in \eqref{qrt} that
\[
\md_v P(Y)<0
\]
when $t_1+t_2<n'$, and also it is easy to
directly verify that
\[
\md_v P(U(t_i))<0
\]
when $t_i\le 0$.
This shows \eqref{qw} and leads as before to
\eqref{aut} and \eqref{aux}.

\item[\bf (ii)]  $L=T_{m,n}$. This is easier.

If $\mu<m$,
use the sublink argument to establish that $K_i(t_i)$ must
be Bennequin sharp. This yields the restriction in
Theorem \ref{th75},
\[
t_i\ge \left(1-\frac{m}{\mu}\right)\frac{n}{\mu}+
\frac{m}{\mu},
\]
by using the result for torus knots in Theorem \ref{TTK}.

If $\mu=m$, we still need $t_i\ge 1$ (and not
only $t_i\ge 0$), since any sublink $U(t_i)$ of $M$
banding an unknotted component of $L$ cannot
yield a split unknot of $M$.

\end{enumerate}

\vspace{-7mm} 
\qed

\medskip 

\begin{corr}\label{CPU}
Assume $\mu<m$, i.e., $T_{m,n}$ is non-pure.
Then $L=T_{m,n}(t_1,\dots,t_\mu)$ or $!T_{m,n}(t_1,\dots,t_\mu)$
is \sp\ if and only if it is \qp.
\end{corr}

\proof
We have $\chi_4(L)=0$, since the components of $T_{m,n}$ are knotted
(and not slice) and none bounds a disk even in $B^4$.
Thus 
\vspace{-3mm} 
\begin{align}
\nonumber \mbox{\qp}\So 
\mbox{Slice-Bennequin-sharp}\So
\mbox{Bennequin-sharp}\So
\mbox{\sp}\,.
\tag*{\qed}
\end{align}

\medskip 

When $L=T_{m,n}$ for $\mu=m$, and $K_i=U$, then $M=L(t_1,\dots,t_\mu)$
for $t_i=0$ has 
\vspace{-2mm} 
\[
\chi_4(M)=2\,.
\]
(Even if there are multiple $K_i=U$
with $t_i=0$, one cannot span more than two slice disks into
unknotted components of $M$ due to linking number reasons.)
But 
\vspace{-2mm} 
\begin{eqn}\label{c44}
\chi_4(M)>\chi(M)
\end{eqn}
breaks down the previous proof.
This suggests how to find the next example, showing that
Corollary \ref{CPU} is indeed false for pure torus links.

\begin{example}\label{x22}
The link $M=T_{2,2}(0,1)$ has the 4-braid representative
\[
[-1\ -2\ -3\ -2\ 1\ 3\ -2\ \ul{3}\ 2\ -1\ 2\ 1\ 3\ 2]\,.
\]
It is \qp, as can be seen when deleting the underlined
letter and showing that the remainder of the word is conjugate to
the last letter $\sg_2$.
Thus $M$ is \qp, but because of \eqref{c44},
certainly not strongly so.
\end{example}

One can, of course, in many cases still use the HOMFLY-PT polynomial
to obstruct to quasipositivity of $M$ through explicit
calculations. But this becomes too hard to control in the
general form. (When $\mu=1$, then of course this troublesome
case reduces to the trivial case.) When $t_i\ne 0$ whenever
$K_i=U$, then the argument for the corollary can still be used,
though.

\medskip 
Similarly to Remark \ref{Cve}, looking at sublinks is
completely useless in studying quasipositivity
(see \cite{BodeDennis}).

\medskip 
Even beyond the auxiliary framings, the difficulties
with pure torus links here are nothing unexpected.
One should by no means assume that the equivalence of
\QP\ and \SP\ is something natural or even common.
There are myriads of instances where it fails.
We highlight that even for knots $K$, these
properties of $C_2(K,t)$ do not coincide
(see \cite{part3}),
and thus such pathologies can only worsen with more components.
Also here, the question when our links $M$ are \qp\ %
remains at least partially unresolved.

\subsubsection{The Baker-Motegi problem\label{SBMP}}

A related problem on \SP{}, raised by Baker and
(independently) Motegi is as follows.

\begin{prob}(Baker-Motegi)\label{PBM}
Assume a link $M$ in \sp. Is every maximal $\chi$
surface of $M$ \sp?
\end{prob}

For instance, this problem was resolved for
canonical surfaces in \cite{benn}, but our surfaces
here are surely highly non-canonical. The problem 
was not discussed in \cite{part2} since, when $L$ is
a knot, the maximal $\chi$ surface of $M$ is unique
(which makes the problem trivial).
However, Example \ref{999} shows that links $L$
are far more subtle. We can now state the following.

\begin{theorem}\label{thmm}
Assume $L$ is a torus link (positive or negative) and
$M=L(t_1,\dots,t_\mu)$ is \sp. Then every
maximal $\chi$ surface of $M$ is \sp.
\end{theorem}

\proof
Let $S$ be a maximal $\chi$ surface of $M$.
Since $M$ bounds a collection of annuli, but has
no split unknotted components, $\chi(M)=0$,
and thus $S$ is a collection of annuli.
This means that $S$ determines a banding structure of
$M=L(t_1,\dots,t_\mu)$.

By Rudolph \cite{Rudolph4}, $S$ is a braided
surface, and thus arises from the grid-band construction
on a grid diagram $D'$ of $L$, \em{but with possibly
negative bands}.

Let $t_i'=\lm(K_i')$ for the components $K_i'$
of the grid diagram $D'$. 
Obviously $t_i'-t_i$ is the number of negative
bands of component $K_i'$, which are smashed into
vertical segments by reversing the grid-band construction.

We correct now the negative bands by
introducing a \em{beaded grid diagram}
(and extending grid-band construction to this
slightly more general case).
On each horizontal segment a bead with an
integer label $\xi$ is allowed, meaning that
at this point the band in the grid-band construction
(now mandatorily positive)
experiences $\xi$ full twists (positive, or $-\xi$
negative if $\xi<0$). Obviously a bead with zero
label can be deleted (or created), and beads on
the same vertical segment can be joined
(through an isotopy of the surface $S$).
\[
\begin{tikzpicture}[scale=0.6, line cap=round, line join=round]

\draw[black,line width=.8pt] (-6,5) -- (-6,2);

\draw[->, line width=1.5pt] (-4.5,3.5) -- (-3,3.5)
      node[midway] {};
\draw[black,line width=1pt] (-6,3) -- (-6,3)
      node[midway, circle, fill=black, inner sep=2pt, label=right:$-1$]{};

\draw[black,line width=1pt] (-6,4) -- (-6,4)
      node[midway, circle, fill=black, inner sep=2pt, label=right:$-1$]{};


\draw[black,line width=.8pt] (-2,5) -- (-2,2);
\draw[black,line width=1pt] (-2,3.5) -- (-2,3.5)
      node[midway, circle, fill=black, inner sep=2pt, label=right:$-2$]{};

\end{tikzpicture}
\]
For instance, the following modifies the
positive stabilization as an isotopy of
$S$.
\begin{eqn}\label{pbs}
\begin{tikzpicture}[scale=0.4, line cap=round, line join=round]

\draw[black,line width=.8pt] (-6,6) -- (-6,2);

\draw[->, line width=1.5pt] (-4.5,4) -- (-3,4)
      node[midway] {};

\draw[black,line width=0.8pt] (-1,2) -- (-1,4);
\draw[black,line width=1pt] (-1,4) -- (1,4);
\draw[black,line width=1pt] (1,4) -- (1,6);
\draw[black,line width=1pt] (1,5) -- (1,5)
      node[midway, circle, fill=black, inner sep=2pt, label=right:$-1$]{};

\end{tikzpicture}
\end{eqn}

Note that it is an isotopy of the surface $S$
to move a bead from top to bottom of a vertical
segment (a flype at the band crossing),
and similarly to move beads across
horizontal segments of the same component of $L$
(flypes at the braid string disk).

Now, by Lemma \ref{lmuq}, the tuple 
$(t_1,\dots,t_\mu)$ is unique up to
permutations, and thus must satisfy
the restrictions of Theorem \ref{th76}.

By the proof of Theorem \ref{th76},
there must exist a sequence of
Cromwell moves \cite{AHT} turning our above grid diagram $D'$
into one $D$ whose components $K_i$ have
$t_i=\lm(K_i)$.

But this sequence of Cromwell moves
obviously extends to a sequence of Cromwell moves
of beaded diagrams: each time a Cromwell
move changes the $TB$ invariant of a component $C$,
a bead can be put on some of the vertical segments
of $C$ to correct for that change. 

These Cromwell moves of beaded diagrams
then yield isotopies of
the surface $S$. (The beaded positive 
stabilization \eqref{pbs} above
is an example.) And at the end 
(the labels of) all beads on every component
of $D$ must cancel out. Thus $S$ is
isotopic to the (\sp) surface that
is obtained by the grid-band construction 
applied on $D$, \em{now with positive bands
only}. \qed

\begin{remark}
Note that, when combined with \S\ref{MSS},
this proof yields the following sharper property
than the results there: if $M=T_{m,n}(t_1,\dots,t_\mu)$
or $M=!T_{m,n}(t_1,\dots,t_\mu)$ is \sp, then \em{every}
\sp{} surface of $M$ is realizable on a
minimal string braid of $M$.
\end{remark}

It is tempting to speculate (see \cite{part2})
whether this property holds for any \sp\ link $M$,
but it appears unlikely.

\subsection{\label{SWH}Whitehead doubled links}

Another application returns to the Whitehead
doubling (see beginning of \S\ref{STB}). Let us write
\vspace{-2mm} 
\begin{eqn}\label{ML}
M=L(t_1^{\pm},\dots,t_\mu^{\pm})
\end{eqn}

\vspace{-2mm} 
\noindent by indicating both
framing and clasp for each component of $L$.
Thus $L(t_1^+)=W_+(L,t_1)$ and $L(t_1^-)=W_-(L,t_1)$
in the previous notation.

We only state the following theorem.
Its proof is a longer (and tricky) amalgamation
of further tools (including cut-and-paste arguments and
concordance invariants), and goes slightly beyond
the style (and volume) of this paper.
(It is available upon request). 

\begin{theorem}\label{CRW}
The link $M=L(t_1^{\pm},\dots,t_\mu^{\pm})$ is \sp{}
if and only if all clasps are positive, and
$(t_1,\dots,t_\mu)$ satisfy
the restrictions of Theorem \ref{th75}
for $L=T_{m,n}$ resp.\ Theorem \ref{th76} for $L=!T_{m,n}$.
\end{theorem}

The following, now less startling, example
warns again that understanding the \QP\ of 
the links $M$ is likely very difficult,
even when $L$ is a torus link.

\begin{example}\label{x23}
Example \ref{x22} can be modified with some care to yield
a 4-braid representative of $M=T_{2,2}(0^+,1^+)$,
\vspace{-2mm} 
\[
[-1\ -2\ -3\ -2\ 1\ 3\ -2\ 3\ 3\ 2\ -1\ 2\ 2\ 1\ 3\ 2]\,.
\]

\vspace{-1mm} 
\noindent 
This link (which consists of an unknot and a right-hand
trefoil component) is thus \qp\ and $\chi_4(M)=0$.
But the proof of Theorem \ref{CRW} does yield that 
for \sp\ $M$,
\vspace{-2mm} 
\begin{eqn}\label{hi}
\chi(M)=-\mu\,.
\end{eqn}

\vspace{-2mm} 
\noindent 
With $\mu=2$, this would again give \eqref{c44}.
Therefore (even without invoking Theorem \ref{CRW}
directly), we see that $M$ cannot be \sp.
\end{example}

\begin{remark}
With a similar combination of a variant of
Corollary \ref{lsrc} and the Bennequin bound
\eqref{eight}, it is not too difficult to argue
that the links $M$ of Theorem \ref{CRW} have a
minimal string \sp{} surface.
(Use \eqref{hi} as in \eqref{eight}.)
However, the resolution of Problem \ref{PBM}
(i.e., an analogue of Theorem \ref{thmm}) remains beyond reach.
\end{remark}

\subsection{Generalizing the $\ell$-invariant for arbitrary links\label{CAL}}

Here we extend some of the previous ideas mostly used
for torus links to general links $L$.
For links, far more technical considerations need
to be made than for knots.

\smallskip 
\noindent {\bf Fundamental assumption:} $L$ has no trivial split components.

\begin{defn}
Assume $L$ has numbered components $K_i$ for $i=1,\dots,\mu$.
The \emph{framing cone} of $L$ is
\vspace{-2mm} 
\[
\Omega(L):=\,\{ (t_1,\dots,t_\mu)\ :\ \mbox{$L(t_1,\dots\,t_\mu)$ is \sp}\,\}\,.
\]

\vspace{-2mm} 
\noindent 
For reasons already explained (and with the above fundamental
assumption),
\vspace{-1mm} 
\[
\Omega(L)=\,\left\{ (t_1,\dots,t_\mu)\ :\ \parbox{5.4cm}{There is
a Legendrian embedding $\cL$ of $L$ with $(\lm(\cK_i))_{i=1}^\mu=(t_i)_{i=1}^\mu$}\ %
\right\}\,.
\]
\end{defn}
Because of positive stabilization, 
if $(t_1,\dots,t_\mu)\in \Omega(L)$, then
\vspace{-2mm} 
\[
\biggl(\,\prod_{i=1}^\mu [t_i,\iy)\! \biggr)\cap \bZ^\mu\sS \Omega(L)\,.
\]

\vspace{-2mm} 
\noindent 
Let 
\vspace{-2mm} 
\[
C(t_1,\dots,t_\mu):=\biggl(\,\prod_{i=1}^\mu [t_i,\iy)\! \biggr)\cap \bZ^\mu\,.
\]

\begin{defn}
When
\vspace{-1mm} 
\[
\Omega(L)=\bigcup_j\, C(t_{1,j},\dots,t_{\mu,j})\,,
\]
so that the union of no proper subfamily of
$\{C(t_{1,j},\dots,t_{\mu,j})\}_j$ covers $\Omega(L)$,
we say that $(t_{1,j},\dots,t_{\mu,j})$ are \em{corner
framings} of $L$.
\end{defn}

The set of corner framings of $L$ is uniquely determined
up to permutation (as easy but slightly tedious exercise).

\begin{example}
Theorems \ref{th75} and \ref{th76} say that $T_{m,n}$
has the single corner framing $(\ap,\dots,\ap)$ for
\begin{eqn}\label{app}
\ap=\left(1-\frac{m}{\mu}\right)\frac{n}{\mu}+\frac{m}{\mu}\,.
\end{eqn}
When $m\nmid n$, then
$!T_{m,n}$ has the single corner framing $(\gm,\dots,\gm)$
for
\[
\gm=\frac{mn}{\mu^2}\,.
\]

\vspace{-2mm} 
\noindent 
Otherwise, for $k=n/m$, it has a set of corner framings
\vspace{-2mm} 
\begin{eqn}\label{aux'}
(u,2k-u,\dots,2k-u)\,,
\end{eqn}

\vspace{-2mm} 
\noindent 
for any $1\le u\le k$, plus (if $u<k$, cyclic) permutations
thereof.
\end{example}

One can easily argue (using the existence of an admissible
framing; see Definition \ref{dfad} and Example \ref{199}) that
corner framings for every fixed link $L$ are finitely
many. But even a (pure) torus link can thus have an
arbitrarily large number of them. Another series of
examples are connected sums of Whitehead links.
(We were aware of them based on Remark
\ref{4cu}, before identifying the pure torus links.)

\smallskip 
A further consequence of the auxiliary framings
\eqref{aux'} (for $u<k$) is that corner
framings need not realize the minimal total
invariant $\lm(L)$ from \eqref{dflm}. 

We do not know whether under some extra
assumptions (like $L$ being alternating\footnote{%
Note that we do need $m>2$ for a torus link in
order this scenario to occur.}), such peculiarity
can be avoided. However, there is enough evidence
that no simple description of corner framings will
be available for most links $L$.

\begin{defn}\label{dfad}
Let us say that
\vspace{-1mm} 
\[
\phi=(t_1,\dots,t_\mu)\; \succeq \; \phi'=(t_1'\dots,t_\mu')
\]
if $t_i\ge t'_i$ for all $i=1,\dots,\mu$, and we call
$\phi$ a \em{stabilization} of $\phi'$.
We say that a framing $(c_1',\dots,c_\mu')$ of $L$
is \em{admissible} if $(c_1',\dots,c_\mu')\preceq (e_1,\dots,e_\mu)$
for every corner framing $(e_1,\dots,e_\mu)$ of $L$ or,
in other words, if
\vspace{-1mm} 
\[
\Omega(L)\sS C(c_1',\dots,c_\mu')\,.
\]
\end{defn}

\begin{example}
If $\mu=1$ (so $L$ is a knot) and assuming
\vspace{-2mm} 
\be\label{xxx}
\lm(U)=1\,,
\ee

\vspace{-2mm} 
\noindent 
then $\phi=(t)$
is admissible if and only if $t\le \lm(L)$.
\end{example}

\begin{example}\label{199}
Note in particular that $\lm(L):=(\lm(K_1),\dots,\lm(K_\mu))$
is always admissible, where we are also allowed
to set \eqref{xxx} due to the fact the we always consider
\sp{} band representations of $U(t_i)$ \em{with bands}.
(Here the assumption also enters that $L$ as no trivial
split components.) Consequently also all $\phi\preceq \lm(L)$
are admissible.
\end{example}

In some attempt to generalize the reasoning for $L=T_{m,n}$,
we formulate some ideas that lead to a different
version of the $\ell$-invariant for links.

\begin{defn}
\begin{itemize} 
\item[\bf (i)] We say that $(\vn=L_0,L_1,\dots,L_\mu)$ is a \em{flag
of sublinks} of $L$, if $L_\mu=L$ and each $L_{i-1}$
is obtained from $L_i$ be deleting exactly one component
of $L_i$.
\vspace{-2mm} 
\item[\bf (ii)] We adapt this notion to bandings of $L$,
by saying $M_{i-1}=L_{i-1}(c_{i-1,1},\dots,c_{i-1,i-1})$
is obtained from $M_i=L_i(c_{i,1},\dots,c_{i,i})$
by deleting the 2 components of $K_{i,j_i}(t_{i,j_i})$,
for some $1\le j_i\le i$, so that 
\[
(c_{i-1,1},\dots,c_{i-1,i-1})\,=\,
(c_{i,1},\dots,c_{i,j_i-1},c_{i,j_i+1},\dots,c_{i,i})\,.
\]
\vspace{-2mm}
\item[\bf (iii)] We say that a flag $(\vn=M_0,\dots,M_\mu=M)$ of banded
sublinks is \em{$P$-increasing} if
\begin{eqn}\label{811}
\Md_v P(M_i)>\Md_v P(M_{i-1})\,,
\end{eqn}
for each $i$, with the stipulation $\Md_vP(\vn):=0$.
\end{itemize} 
\end{defn}

\medskip 
While the condition looks technical, keep in mind
that it is a finite number of inequalities \eqref{811}.
It can be easily seen that their number is the
same as the number of edges in the $1$-skeleton of
the $\mu$-dimensional cube, which is $\mu2^{\mu-1}$.

\begin{theorem}
Let $L$ have an admissible framing $\phi$, so that each flag
of banded links of $M=L(\phi)$ is $P$-increasing.
Then 
\[
\ell_0(L):=\mathop{\min}\limits_{\kp\,\scbox{corner framing}} \frac{1+\Md_v P(L(\kp))}{2}
\]
satisfies
\[
a(L)\ge \ell_0(L)\,.
\]
\end{theorem}

\proof[Proof sketch] We only outline the argument, as it mostly
repeats reasoning rolled out for the torus links.

The idea is to use sublinks as in the proof of Corollary
\ref{lsrc}, and the Bennequin inequality constraint \eqref{eight}.
The preceding remark also clarifies that
$\kp$ are finitely many, thus no problem occurs
building the minimum.

(The condition of $M$ being $P$-increasing implies
that the expressions minimized over are positive.) \qed

\begin{example}
We write $\ell_0$ to emphasize that
for knots $K$, not always $\ell_0=\ell$.
For example $\ell(10_{132})=8$, but with
$\lm(10_{132})=1$ (which, in passing by, also fixes
our mirroring for the knot), we have $\Md_v P(C_2(10_{132},1))=17$,
thus $\ell_0(10_{132})=9$.
\end{example}

However, it is true that 
\begin{eqn}\label{ell1}
\ell_0(K)\ge \ell(K)\,,
\end{eqn}
and we invite the reader to think about it
(see below Example \ref{xd1}).

\medskip 

A more practical modification, which does not require
one to know the corner framings, is the following.

\begin{theorem}\label{th89}
Let $L$ have an admissible framing $\phi$, so that each flag
of banded links of $M=L(\phi)$ is $P$-increasing.
Then 
\[
\ell_\phi(L):=1+\frac{\Md_v P(L(\phi))-\min(1,\md_v[P(L(\phi))]_{z^{1-2\mu}})}{2}
\]
satisfies
\[
a(L)\ge \ell_\phi(L)\,.
\]
\end{theorem}

\proof[Proof sketch] This uses that $[P(M)]_{z^{1-2\mu}}$ just shifts with
change of linking numbers (see the proof of Theorem \ref{prt}).
\qed

\newpage 

\begin{example}\label{xd1}
This bound clearly depends on $\phi$. For example,
if $L=10_{132}$ and $\phi=(0)$, then $\ell_\phi(L)=\ell(L)=8$,
but as we saw, if $\phi=(1)$, then $\ell_\phi(L)=9$.
\end{example}

We leave it as an exercise to the reader to see that
when $L=K$ is a knot, and
\vspace{-2mm} 
\[
\phi=(\th(K))\,,
\]

\vspace{-2mm} 
\noindent with \eqref{9811}, which can be rephrased as
\vspace{-2mm} 
\[
\th(K)=\min\{\,t\in\bZ\,:\,\md_v P(C_2(K,t))>0\,\}\,,
\]
then $\ell_\phi(K)=\ell(K)$ recovers the $\ell$-invariant
\eqref{defli} of knots. (In the above example $\th(10_{132})=0$.)
This argument also easily shows that 
\begin{eqn}\label{ell2}
\ell(L)\le \ell_\phi(K)\,,
\end{eqn}
for a knot $K$, whatever admissible framing $\phi$ is chosen.
This relates to, and in fact also explains, \eqref{ell1}.

\begin{example}
A final example is to (briefly) revisit the torus links
through Theorem \ref{th89}. When $L=T_{m,n}$, then
$\lm(K_i)=\ap$ in \eqref{app}. Thus $\phi=(\dl,\dots,\dl)$
from \eqref{dll} is admissible. We checked (essentially
because of \eqref{qqq}) that $\phi$ gives a banding with
$P$-increasing flags. Then from Theorem \ref{PHT} we have
\vspace{-2mm} 
\[
\Md_v P(L(\phi))=2n-1\,,\quad
\md_v [P(L(\phi))]_{z^{1-2\mu}}=1-2m\,,\quad
\]

\vspace{-2mm} 
\noindent
yielding $\ell_\phi(L)=m+n$, as before.
\end{example}

} 

\section*{Acknowledgement}

We are grateful to Andrey Morozov for a useful discussion. The work of A.M. was partially funded within the state
assignment of the Institute for Information Transmission Problems of RAS, was partly supported by the grant
of the Foundation for the Advancement of Theoretical Physics and Mathematics ``BASIS'' and by Armenian SCS grants 24WS-1C031. The work of H.S. and V.K.S. is supported by Tamkeen UAE under the 
NYU Abu Dhabi Research Institute grant {\tt CG008}. 

\bibliographystyle{amsplain}

\end{document}